\documentclass[a4paper,leqno]{amsart}

\usepackage{latexsym}
\usepackage[english]{babel}
\usepackage{fancyhdr}
\usepackage[mathscr]{eucal}
\usepackage{amsmath}
\usepackage{mathrsfs}
\usepackage{amsthm}
\usepackage{amsfonts}
\usepackage{amssymb}
\usepackage{amscd}
\usepackage{bbm}
\usepackage{graphicx}
\usepackage{subcaption}
\usepackage{graphics}
\usepackage{latexsym}
\usepackage{color}

\usepackage{pifont}
\usepackage{booktabs} 

\newcommand{\ud}{\mathrm{d}}

\newcommand{\ii}{\mathrm{i}}
\newcommand{\cH}{\mathcal{H}}
\newcommand{\ran}{\mathrm{ran}}

\theoremstyle{plain}
\newtheorem{theorem}{Theorem}[section]
\newtheorem{lemma}[theorem]{Lemma}
\newtheorem{corollary}[theorem]{Corollary}
\newtheorem{proposition}[theorem]{Proposition}

\theoremstyle{definition}

\newtheorem{remark}[theorem]{Remark}
\newtheorem{example}[theorem]{Example}

\numberwithin{equation}{section}

\begin{document}

\title[On Krylov solutions to infinite-dimensional inverse linear problems]
{On Krylov solutions to infinite-dimensional inverse linear problems}
\author[N.~Caruso]{Noe Caruso}
\address[N.~Caruso]{International School for Advanced Studies -- SISSA \\ via Bonomea 265 \\ 34136 Trieste (Italy).}
\email{ncaruso@sissa.it}
\author[A.~Michelangeli]{Alessandro Michelangeli}
\address[A.~Michelangeli]{International School for Advanced Studies -- SISSA \\ via Bonomea 265 \\ 34136 Trieste (Italy).}
\email{alemiche@sissa.it}
\author[P.~Novati]{Paolo Novati}
\address[P.~Novati]{Universit\`{a} degli Studi di Trieste \\ Piazzale Europa 1 \\ 34127 Trieste (Italy).}
\email{novati@units.it}


\begin{abstract}
%
We discuss, in the context of inverse linear problems in Hilbert space, the notion of the associated infinite-dimensional Krylov subspace and we produce necessary and sufficient conditions for the Krylov-solvability of a given inverse problem, together with a series of model examples and numerical experiments.
\end{abstract}

\date{\today}

\subjclass[2000]{}
\keywords{inverse linear problems, infinite-dimensional Hilbert space, ill-posed problems, orthonormal basis discretisation, bounded linear operators, self-adjoint operators, Krylov subspaces, cyclic operators, Krylov solution, GMRES, CG, LSQR}

%

\maketitle


\section{Introduction and outlook}\label{intro}

Krylov subspace methods constitute a wide class of efficient numerical schemes for finite-dimensional inverse linear problems, even counted among the `Top 10 Algorithms' of the 20th century \cite{Dongarra-Sullivan-Best10-2000,Cipra-SIAM-News}.

Whereas this framework is by now classical and deeply understood for finite-dimensional inverse problems (see, e.g., the monographs \cite{Saad-2003_IterativeMethods,Liesen-Strakos-2003} or also \cite{Saad-1981}), it is instead less explored -- and surely lacks a systematic study -- in the infinite-dimensional case \cite{Karush-1952,Daniel-1967,Kammerer-Nashed-1972,Nemirovskiy-Polyak-1985,Winther-1980,Herzog-Ekkehard-2015}.

In this work we focus on the general setting of infinite-dimensional inverse linear problems that are solved by means of finite-dimensional truncations taken with respect to a basis of the associated Krylov subspace, and we investigate the possibility that the solution can be indeed well approximated by vectors in the Krylov subspace.

To fix the nomenclature and the notation, let us consider an \emph{inverse linear problem} in Hilbert space, namely the problem, given a Hilbert space $\cH$, a linear operator $A$ acting on $\cH$, and a vector $g\in\cH$, to determine the solution(s) $f\in\cH$ to the linear equation
\begin{equation}\label{eq:gen}
 Af\;=\;g\,.
\end{equation}
We shall say that: \eqref{eq:gen} is \emph{solvable} if a solution $f$ exists, namely if $g\in\mathrm{ran}A$; \eqref{eq:gen} is \emph{well-defined} if additionally the solution $f$ is unique, i.e., if $A$ is also injective (in which case one refers to  $f$ `\emph{exact}' solution); \eqref{eq:gen} is \emph{well-posed} if there exists a unique solution that depends continuously (i.e., in the norm of $\cH$) on the datum $g$, equivalently, that $g\in\mathrm{ran} A$ and $A$ has bounded inverse on its range.

Although well-defined inverse linear problems are in a sense trivial theoretically, as the existence and uniqueness of the solution is not of concern, a crucial numerical issue is the control of the \emph{truncation} to the finite-dimensional space in which approximate solutions are to be computed. Obviously this refers to the case when $\dim\cH=\infty$ and $A$ is a genuine infinite-dimensional operator on $\cH$: by this we mean, as customary \cite[Sect.~1.4]{schmu_unbdd_sa}, that $A$ is \emph{not} reduced as $A=A_1\oplus A_2$ by an orthogonal direct sum decomposition $\cH=\cH_1\oplus\cH_2$ with $\dim\cH_1<\infty$, $\dim\cH_2=\infty$, and $A_2=\mathbb{O}$.

In the framework of standard Galerkin and Petrov-Galerkin methods \cite{Ern-Guermond_book_FiniteElements,Quarteroni-book_NumModelsDiffProb}, typically developed for partial differential 
operators, the well-posedness of the problem \eqref{eq:gen} is ensured by various classical conditions (in practice some kind of coercivity of $A$), such as the Banach-Ne\u{c}as-Babu\v{s}ka Theorem or the Lax-Milgram Lemma \cite[Chapter 2]{Ern-Guermond_book_FiniteElements}. Analogous conditions guarantee the well-posedness of the truncated problems, and in order for the finite-dimensional solutions to converge strongly in the infinite-dimensional limit, one requires stringent yet often plausible conditions \cite[Sect.~2.2-2.4]{Ern-Guermond_book_FiniteElements}, \cite[Sect.~4.2]{Quarteroni-book_NumModelsDiffProb} both on the truncation spaces, that need approximate suitably well the ambient space $\cH$ (`\emph{approximability}', thus the interpolation capability of finite elements),
and on the behaviour of the reduced problems, that need admit solutions that are uniformly controlled by the data 
(`\emph{uniform stability}'), and that are suitably good approximate solutions of the original problem (`\emph{asymptotic consistency}'), together with some suitable boundedness of the problem in appropriate topologies (`\emph{uniform continuity}').

For non-differential inverse problems, for example when $A$ is a compact or a generic bounded operator with a bad-behaving inverse (e.g., when $A$ is non-coercive), as is often the case when $A$ is an integral operator, the solvability or singularity of the truncated problems and the error analysis in the infinite-dimensional limit are being studied as well \cite{Gasparo-Papini-Pasquali-2008,CMN-truncation-2018}.

In this respect, Krylov subspace methods are a class of algorithms where approximate solutions to \eqref{eq:gen} are sought among the linear combinations of the vectors $g,Ag,A^2g,\dots$ which span the so-called `\emph{Krylov subspace}' $\mathcal{K}(A,g)$ associated with $A$ and $g$.

The infinite-dimensionality of the underlying Hilbert space $\cH$ comes with a load of new issues, starting from the very definition of the Krylov vectors $A^k g$ if $A$ is unbounded \cite{CM-Nemi-unbdd-2019}. Even when $A$ is everywhere defined and bounded, and hence $\mathcal{K}(A,g)$ is well-defined, it may well happen that $\mathcal{K}(A,g)$ is not dense in $\cH$, thus preventing the truncation spaces to have that approximability feature which, as mentioned above, is a typical assumption for (Petrov-)Galerkin schemes.

Among such potential difficulties, the first crucial question is whether the solution(s) to \eqref{eq:gen} can be well approximated by vectors in $\mathcal{K}(A,g)$, say, whether they belong to the closure $\overline{\mathcal{K}(A,g)}$ taken in the $\cH$-norm topology. In the affirmative case, the Krylov subspace is a reliable space for the approximants of the exact solution(s): we shall refer to such an occurrence by saying that the problem \eqref{eq:gen} is `\emph{Krylov-solvable}' and a solution $f$ to \eqref{eq:gen} such that $f\in\overline{\mathcal{K}(A,g)}$ will be referred to as a `\emph{Krylov solution}'. 

Additional relevant questions then arise, for example in the presence of a multiplicity of solutions some may be Krylov and others may not.

In the present work we investigate an amount of mathematical aspects of a (genuinely) infinite-dimensional bounded inverse linear problem in Hilbert space with respect to the underlying Krylov subspace.

After fixing the natural generalisation of the Krylov subspace in infinite dimensions (Sect.~\ref{sec:Krylov}), we address the general question of the Krylov solvability. Through several paradigmatic examples and counter-examples we show the typical occurrences where such a feature may hold or fail.

Most importantly, we demonstrate necessary and sufficient conditions, for certain relevant classes of bounded operators, in order for the solution to be a Krylov solution (Sect.~\ref{sec:Krylov_soll}).

To this aim, we identify a somewhat `intrinsic' notion associated to the operator $A$ and the datum $g$, a subspace that we call the `\emph{Krylov intersection}', that turns out to qualify the operator-theoretic mechanism for the Krylov-solvability of the problem.

We observe that for the study case that is most investigated in the previous literature of infinite-dimensional Krylov subspaces, namely the self-adjoint bounded inverse linear problems, this mechanism takes a more explicit form, that we shall refer to as the `\emph{Krylov reducibility}' of the operator $A$.

Last, in the concluding part, Section \ref{sec:numerics}, we investigate the main features discussed theoretically through a series of numerical tests on inverse problems in infinite-dimensional Hilbert space, suitably truncated and analysed by increasing the size of the truncation.

%

\bigskip

\textbf{General notation.} Besides further notation that will be declared in due time, we shall keep the following convention. $\cH$ denotes a complex Hilbert space, that is assumed to be separable throughout this note, with norm $\|\cdot\|_{\cH}$ and scalar product $\langle\cdot,\cdot\rangle$, anti-linear in the first entry and linear in the second. Bounded operators on $\cH$ shall be tacitly understood as linear and everywhere defined: they naturally form a space, denoted with $\mathcal{B}(\cH)$, with Banach norm $\|\cdot\|_{\mathrm{op}}$, the customary operator norm.
$\mathbbm{1}$ and $\mathbbm{O}$ denote, respectively, the identity and the zero operator, meant as finite matrices or infinite-dimensional operators depending on the context. An upper bar denotes the complex conjugate $\overline{z}$ when $z\in\mathbb{C}$, and the norm closure $\overline{\mathcal{V}}$ of the span of the vectors in $\mathcal{V}$ when $\mathcal{V}$ is a subset of $\cH$. For $\psi,\varphi\in\cH$, by $|\psi\rangle\langle\psi|$ and $|\psi\rangle\langle\varphi|$ we shall denote the $\cH\to\cH$ rank-one maps acting respectively as $f\mapsto \langle \psi, f\rangle\,\psi$ and $f\mapsto \langle \varphi, f\rangle\,\psi$ on generic $f\in\cH$. For identities such as $\psi(x)=\varphi(x)$ in $L^2$-spaces, the standard `for almost every $x$' declaration will be tacitly understood.

\section{Krylov subspace in infinite dimensional Hilbert space}\label{sec:Krylov}

\subsection{Definition and generalities}~

As well known, given a $d\times d$ (complex) matrix $A$ and a vector $g\in\mathbb{C}^d$, the $N$-th order Krylov subspace associated with $A$ and $g$ is the subspace
\begin{equation}
  \mathcal{K}_N(A,g) \;:=\;  \mathrm{span} \{ g, Ag, \dots , A^{N-1}g\}\;\subset\;\mathbb{C}^d\,.
\end{equation}

Clearly, $1\leqslant\dim\mathcal{K}_N(A,g)\leqslant N$, and there always exists some $N_0\leqslant d$ such that all $N$-th order spaces $\mathcal{K}_N(A,g)$ are the same whenever $N\geqslant N_0$: one then refers to \emph{the} Krylov subspace associated with $A$ and $g$ as the maximal subspace $\mathcal{K}_{N_0}(A,g)$.

This notion has a natural generalisation to a Hilbert space $\cH$ with $\dim\cH=\infty$, an everywhere defined, bounded linear operator $A:\cH\to\cH$, and a vector $g\in \cH$: clearly, unlike the finite-dimensional case, it may happen that  $\sup_N \dim\mathcal{K}_N(A,g)=\infty$.

The Krylov subspace associated with $A$ and $g$ is then defined as
\begin{equation}\label{eq:defKrylov}
  \mathcal{K}(A,g) \;:=\;  \mathrm{span} \{ A^kg\,|\,k\in\mathbb{N}_0\}\,,
\end{equation}
a definition that applies to the finite-dimensional case too. In fact \eqref{eq:defKrylov} makes sense also when $A$ is unbounded, provided that $g$ simultaneously belongs to the domains of all the powers of $A$. Yet, in the present discussion we shall stick to bounded operators acting on (possibly infinite-dimensional) Hilbert spaces.

Obviously, when $\dim\mathcal{K}(A,g)=\infty$ the subspace $\mathcal{K}(A,g)$ is not closed in $\cH$. Its closure can either be a proper closed subspace of $\cH$, or even the whole $\cH$ itself.

\begin{example}~

 \begin{itemize}
  \item[(i)] For  the right-shift operator $R$ on $\ell^2(\mathbb{N})$ (Sec.~\ref{sec:Rshift}) and the vector $g=e_{m+1}$ (one of the canonical basis vectors), $\overline{\mathcal{K}(R,e_{m+1})}=\mathrm{span}\{e_1,\dots,e_m\}^\perp$, which is a proper subspace of $\ell^2(\mathbb{N})$ if $m\geqslant 1$, and instead is the whole $\ell^2(\mathbb{N})$ if $g=e_1$.
  \item[(ii)] For the Volterra integral operator $V$ on $L^2[0,1]$ (Sec.~\ref{sec:Volterra}) and the function $g=\mathbf{1}$ (the constant function with value 1), it follows from \eqref{eq:defVolterra} or \eqref{eq:powersVolterra} that the functions $Vg,V^2g,V^3g,\dots$ are (multiples of) the polynomials $x,x^2,x^3,\dots$, therefore $\mathcal{K}(V,g)$ is the space of polynomials on $[0,1]$, which is dense in $L^2[0,1]$.
 \end{itemize}
\end{example}

In purely operator-theoretical contexts, the Krylov subspace $ \mathcal{K}(A,g)$ is customarily referred to as the \emph{cyclic space} of $A$ relative to the vector $g$, the spanning vectors $g$, $Ag$, $A^2g$, $\dots$ form the \emph{orbit} of $g$ under $A$, the density of $ \mathcal{K}(A,g)$ in $\cH$ is called the \emph{cyclicity} of $g$, in which case $g$ is called a \emph{cyclic vector} for $A$, and when $A$ admits cyclic vectors in $\cH$ one says that $A$ is a \emph{cyclic operator}.

For completeness of information, let us recall a few well-known facts about cyclic vectors and cyclic operators \cite{Halmos-HilbertSpaceBook}.
\begin{itemize}
 \item[(I)] In non-separable Hilbert spaces there are no cyclic vectors.
 \item[(II)] The set of (bounded) cyclic operators on a Hilbert space $\cH$ is dense in $\mathcal{B}(\cH)$, with respect to the $\|\,\|_{\mathrm{op}}$-norm, if $\dim\cH<\infty$; instead, it is not dense in $\mathcal{B}(\cH)$ if $\dim\cH=\infty$.
 \item[(III)] The set of cyclic operators on a separable Hilbert space $\cH$ is not closed in $\mathcal{B}(\cH)$. It is open in $\mathcal{B}(\cH)$ if $\dim\cH<\infty$, it is not open if $\dim\cH=\infty$.
 \item[(IV)] If $\dim\cH=\infty$ and $\cH$ is separable, then the set of non-cyclic operators on $\cH$ is dense in $\mathcal{B}(\cH)$ (whereas, instead, the set of cyclic operators is not). 
 \item[(V)] It is not known whether there exists a bounded operator on a separable Hilbert space $\cH$ such that \emph{every} non-zero vector in $\cH$ is cyclic.
 \item[(VI)] The set of cyclic vectors for a bounded operator $A$ on a Hilbert space $\cH$ is either empty or a dense subset of $\cH$ \cite{Geher-1972}. If $g$ is a cyclic vector for $A$, then also $g^{(n)}:=(1-\alpha A)^ng$ for any $|\alpha|\in(0,\|A\|^{-1})$ and for any $n\in\mathbb{N}$, and the $g^{(n)}$'s thus defined span the whole $\cH$. 
 \item[(VII)] A bounded operator $A$ on the separable Hilbert space $\cH$ is cyclic if and only if there is an orthonormal basis $(e_n)_n$ of $\cH$ with respect to which the matrix elements $a_{ij}:=\langle e_i,A e_j\rangle$ are such that $a_{ij}=0$ for $i>j+1$ and $a_{ij}\neq 0$ for $i=j+1$ (thus, $A$ is an upper Hessenberg infinite-dimensional matrix).
\end{itemize}



\subsection{Krylov reducibility and Krylov intersection}~

For given $A\in\mathcal{B}(\cH)$ and $g\in\cH$, there is the orthogonal decomposition
\begin{equation}\label{eq:Krylov_decomposition}
 \cH\;=\;\overline{\mathcal{K}(A,g)}\,\oplus\,\mathcal{K}(A,g)^\perp
\end{equation}
that we shall often refer to as the \emph{Krylov decomposition} of $\cH$ relative to $A$ and $g$. The corresponding Krylov subspace is invariant under $A$ and its orthogonal complement is invariant under $A^*$, that is,
\begin{equation}\label{eq:invariances}
 A\,\mathcal{K}(A,g)\,\subset\,\mathcal{K}(A,g)\,,\qquad A^*\,\mathcal{K}(A,g)^\perp\,\subset\,\mathcal{K}(A,g)^\perp\,.
\end{equation}
The first statement is obvious and the second follows from $\langle A^*w,z\rangle=\langle w,Az\rangle=0$ $\forall z\in \mathcal{K}(A,g)$, where $w$ is a generic vector in $\mathcal{K}(A,g)^\perp$. Owing to the evident relations
\begin{equation}\label{eq:inclusions_closure}
 \begin{split}
 A \mathcal{V}\:\subset\:A\overline{\mathcal{V}}\:&\subset\:\overline{A\mathcal{V}\,}\;=\;\overline{A\overline{\mathcal{V}}\:} \\
 A\overline{\mathcal{V}}\:&=\:\overline{A\mathcal{V}\,}\quad\textrm{if }A^{-1}\in\mathcal{B}(\cH)\,,
 \end{split}
 \end{equation}
valid for any subspace $\mathcal{V}\subset\cH$ and any $A\in\mathcal{B}(\cH)$, \eqref{eq:invariances} implies also
\begin{equation}\label{eq:invariances2}
 A\,\overline{\mathcal{K}(A,g)}\,\subset\,\overline{\mathcal{K}(A,g)}\,.
\end{equation}

A relevant occurrence for our purposes is when the operator $A$ is \emph{reduced} by the Krylov decomposition \eqref{eq:Krylov_decomposition}, meaning that both $\overline{\mathcal{K}(A,g)}$ and $\mathcal{K}(A,g)^\perp$ are invariant under $A$. For short, we shall refer to this feature as the \emph{Krylov reducibility}, or also, to avoid ambiguity, $\mathcal{K}(A,g)$-reducibility. We shall discuss the relevance of this feature in the subsequent sections when we discuss general conditions for Krylov-solvability.

It follows from this definition that if $A$ is $\mathcal{K}(A,g)$-reduced, then $A^*$ is $\mathcal{K}(A,g)$-reduced too, and vice versa, as one sees from the following elementary Lemma, whose proof is omitted.

\begin{lemma}\label{lem:AreducedAstarreduced}
 If $A$ is a bounded operator on a Hilbert space $\cH$ and $\mathcal{V}\subset\cH$ is a \emph{closed} subspace, then properties (i) and (ii) below are equivalent:
 \begin{itemize}
  \item[(i)] $A\mathcal{V}\subset \mathcal{V}$ and $A\mathcal{V}^\perp\subset\mathcal{V}^\perp$;
  \item[(ii)] $A^*\mathcal{V}\subset \mathcal{V}$ and $A^*\mathcal{V}^\perp\subset\mathcal{V}^\perp$.
 \end{itemize}
%
\end{lemma}

\begin{example}\label{example:Kry-red}~
 \begin{itemize}
  \item[(i)] For generic $A\in\mathcal{B}(\cH)$ and $g\in\cH$, $A$ may fail to be $\mathcal{K}(A,g)$-reduced. All bounded self-adjoint operators are surely Krylov-reduced, owing to \eqref{eq:invariances}.
  \item[(ii)] Yet, Krylov reducibility is not a feature of self-adjoint operators only. To see this, let $A,B\in\mathcal{B}(\cH)$ and $\widetilde{A}:=A\oplus B:\cH\oplus\cH\to \cH\oplus\cH$. If $g\in\cH$ is a cyclic vector for $A$ in $\cH$ and $\widetilde{g}:=g\oplus 0$, then $\overline{\mathcal{K}(\widetilde{A},\widetilde{g})}=\cH\oplus\{0\}$, implying that $\mathcal{K}(\widetilde{A},\widetilde{g})^\perp=\{0\}\oplus\cH$. Therefore, $\widetilde{A}$ is $\mathcal{K}(\widetilde{A},\widetilde{g})$-reduced. On the other hand, $\widetilde{A}$ is self-adjoint (on $\cH\oplus\cH$) if and only if so are both $A$ and $B$ on $\cH$.
  \end{itemize}
\end{example}

For \emph{normal operators} we have the following equivalent characterisation of Krylov reducibility.

\begin{proposition}\label{prop:normops_Kryred_Astarg}
 Let $A$ be a bounded normal operator on a Hilbert space $\cH$ and let $g\in\cH$. Then $A$ is $\mathcal{K}(A,g)$-reduced if and only if $A^*g\in\overline{\mathcal{K}(A,g)}$.
\end{proposition}

\begin{proof}
 If $A$ is $\mathcal{K}(A,g)$-reduced, then $\overline{\mathcal{K}(A,g)}$ is invariant under $A^*$ (Lemma \ref{lem:AreducedAstarreduced}), hence in particular $A^*g\in\overline{\mathcal{K}(A,g)}$. Conversely, if $A^*g\in\overline{\mathcal{K}(A,g)}$, then 
 \[
  \overline{\mathcal{K}(A, A^*g)}\;=\;\overline{\mathrm{span}\{A^kA^*g\,|\,k\in\mathbb{N}_0\}}\;\subset\;\overline{\mathcal{K}(A,g)}\,,
 \]
and moreover, since $A$ is normal, $A^*\mathcal{K}(A,g)=\mathcal{K}(A, A^*g)$; therefore (using \eqref{eq:inclusions_closure}),
\[
 A^*\overline{\mathcal{K}(A,g)}\;\subset\;\overline{\mathcal{K}(A, A^*g)}\;\subset\;\overline{\mathcal{K}(A,g)}\,.
\]
The latter property, together with $A^*\,\mathcal{K}(A,g)^\perp\subset\mathcal{K}(A,g)^\perp$ (from \eqref{eq:invariances} above) imply that $A^*$ is $\mathcal{K}(A,g)$-reduced, and so is $A$ itself, owing to Lemma \ref{lem:AreducedAstarreduced}. 
\end{proof}

For $A\in\mathcal{B}(\cH)$ and $g\in\cH$, an obvious consequence of $A$ being $\mathcal{K}(A,g)$-reduced is that $\overline{\mathcal{K}(A,g)}\cap (A\,\mathcal{K}(A,g)^\perp)=\{0\}$. For its relevance in the following, we shall call the intersection
\begin{equation}\label{eq:Krilov_intersection}
\overline{\mathcal{K}(A,g)}\,\cap\, (A\,\mathcal{K}(A,g)^\perp) 
\end{equation}
the \emph{Krylov intersection} for the given $A$ and $g$.

\begin{example}\label{example:Atheta}
The Krylov intersection may be trivial also in the absence of Krylov reducibility. This is already clear for finite-dimensional matrices: for example, taking (with respect to the Hilbert space $\mathbb{C}^2$)
\[
 A_\theta\,=\,\begin{pmatrix}
  1 & \cos\theta \\
  0 & \sin\theta
 \end{pmatrix}\qquad\theta\in(0,{\textstyle\frac{\pi}{2}}]\,,\qquad g\,=\,
 \begin{pmatrix}
  1 \\ 0
 \end{pmatrix},
\]
one sees that $A_\theta$ is $\mathcal{K}(A_\theta,g)$-reduced only when $\theta=\frac{\pi}{2}$, whereas the Krylov intersection \eqref{eq:Krilov_intersection} is trivial for any $\theta\in(0,\frac{\pi}{2})$.
\end{example}

\section{Krylov solutions for a bounded linear inverse problem}\label{sec:Krylov_soll}

\subsection{Krylov solvability. Examples.}~

Let us re-consider the bounded linear inverse problem of the type \eqref{eq:gen}: given $A\in\mathcal{B}(\cH)$ and the datum $g\in\mathrm{ran}A$, one searches for solution(s) $f\in\cH$ to $Af=g$.

The general question we are studying here is when $f$ to $Af=g$ admits arbitrarily close (in the norm of $\cH$) approximants expressed by finite linear combinations of the spanning vectors $A^kg$'s, or equivalently, $f$ belongs to the closure $\overline{\mathcal{K}(A,g)}$ of the Krylov subspace relative to $A$ and $g$.

A solution $f$ satisfying the above property is referred to as a \emph{Krylov solution}.

Informally, we shall use the expression \emph{Krylov solvability} for the feature that a linear inverse problem has a Krylov solution.

\begin{example}\label{example:Krylov_sol}~
\begin{itemize}
 \item[(i)] The self-adjoint multiplication operator $M:L^2[1,2]\to L^2[1,2]$, $\psi\mapsto x\psi$ is bounded and invertible with an everywhere defined bounded inverse. The solution to $Mf=\mathbf{1}$ is the function $f(x)=\frac{1}{x}$. Moreover, $\mathcal{K}(M,\mathbf{1})$ is the space of polynomials on $[1,2]$, hence it is \emph{dense} in $L^2[1,2]$: therefore $f$ is a Krylov solution.
 \item[(ii)] The multiplication operator $M_z:L^2(\Omega)\to L^2(\Omega)$, $f\mapsto zf$, where $\Omega\subset\mathbb{C}$ is a bounded open subset separated from the origin, say, $\Omega=\{z\in\mathbb{C}\,|\,|z-2|<1\}$, is a normal bounded bijection on $L^2(\Omega)$ (Sec.~\ref{sec:mult_annulus}), and the solution to $M_zf=g$ for given $g\in L^2(\Omega)$ is the function $f(z)=z^{-1}g(z)$.  Moreover, $\mathcal{K}(M_z,g)=\{p\,g\,|\,p\textrm{ a polynomial in $z$ on }\Omega\}$. One can see that $f\in\overline{\mathcal{K}(M_z,g)}$ and hence the problem $M_zf=g$ is Krylov-solvable. Indeed, $\Omega\ni z\mapsto z^{-1}$ is holomorphic and hence is realised by a uniformly convergent power series (e.g., the Taylor expansion of $z^{-1}$ about $z=2$). If $(p_n)_n$ is such a sequence of polynomial approximants, then
 \[
  \begin{split}
   \|f-p_n g\|_{L^2(\Omega)}\;&=\;\|(z^{-1}-p_n)g\|_{L^2(\Omega)} \\
   &\leqslant\;\|z^{-1}-p_n\|_{L^\infty(\Omega)}\|g\|_{L^2(\Omega)}\;\xrightarrow[]{n\to\infty}\;0\,.
  \end{split}
 \]
 \item[(iii)] The left-shift operator $L$ on $\ell^2(\mathbb{N}_0)$ (Sec.~\ref{sec:Rshift}) is bounded, not injective, and with range $\mathrm{ran}L=\ell^2(\mathbb{N}_0)$. The solution to $Lf=g$ with $g:=\sum_{n\in\mathbb{N}_0}\frac{1}{n!}e_n$ is $f=\sum_{n\in\mathbb{N}_0}\frac{1}{n!}e_{n+1}$. Moreover, $\mathcal{K}(L,g)$ is \emph{dense} in $\ell^2(\mathbb{N}_0)$ and therefore $f$ is a Krylov solution. To see the density of $\mathcal{K}(L,g)$: the vector $e_0$ belongs to $\overline{\mathcal{K}(L,g)}$ because
 \[
  \begin{split}
   \|k!\, L^k g-e_0\|_{\ell^2}^2\;&=\;\|\textstyle(1,\frac{1}{k+1},\frac{1}{(k+2)(k+1)},\cdots)-(1,0,0,\dots)\|_{\ell^2}^2 \\
   &=\;\sum_{n=1}^\infty\Big(\frac{k!}{(n+k)!}\Big)^2\;\xrightarrow[]{\;k\to\infty\;}\;0\,.
  \end{split}
 \]
 As a consequence, $ (0,\textstyle\frac{1}{k!},\frac{1}{(k+1)!},\frac{1}{(k+2)!},\cdots)=L^{k-1}g-(k-1)!\,e_0\in\overline{\mathcal{K}_g(L)}$,
therefore the vector $e_1$ too belongs to $\overline{\mathcal{K}_g(L)}$, because
\[
 \|k!\,(L^{k-1}g-(k-1)!\,e_0)-e_1\|_{\ell^2}^2\;=\;\sum_{n=1}^\infty\Big(\frac{k!}{(n+k)!}\Big)^2\;\xrightarrow[]{\;k\to\infty\;}\;0\,.
\]
 Repeating inductively the above two-step argument proves that any $e_n\in\overline{\mathcal{K}(L,g)}$, whence the cyclicity of $g$. 
  \item[(iv)] The right-shift operator $R$ on $\ell^2(\mathbb{N})$ (Sec.~\ref{sec:Rshift}) is bounded and injective, with non-dense range, and the solution to $Rf=e_2$ is $f=e_1$. However, $f$ is not a Krylov solution, for $\overline{\mathcal{K}(R,e_2)}=\overline{\mathrm{span}\{e_2,e_3,\dots\}}$. The problem $Rf=e_2$ is not Krylov-solvable.
  \item[(v)] The compact (weighted) right-shift operator $\mathcal{R}$ on $\ell^2(\mathbb{Z})$ (Sec.~\ref{sec:compactRshift-Z}) is normal, injective, and with dense range, and the solution to $\mathcal{R}f=\sigma_1 e_2$ is $f=e_1$. However, $f$ is not a Krylov solution, for $\overline{\mathcal{K}(\mathcal{R},e_2)}=\overline{\mathrm{span}\{e_2,e_3,\dots\}}$. The problem $\mathcal{R}f=\sigma_1e_2$ is not Krylov-solvable.
  \item[(vi)] Let $A$ be a bounded injective operator on a Hilbert space $\cH$ with cyclic vector $g\in\mathrm{ran}A$ and let $\varphi_0\in\cH\setminus\{0\}$. Let $f\in\cH$ be the solution to $Af=g$. The operator $\widetilde{A}:=A\oplus |\varphi_0\rangle\langle\varphi_0|$ on $\widetilde{\cH}:=\cH\oplus\cH$ is bounded. One solution to $\widetilde{A}\widetilde{f}=\widetilde{g}:=g\oplus 0$ is $\widetilde{f}=f\oplus 0$ and $\widetilde{f}\in\cH\oplus\{0\}=\overline{\mathcal{K}(\widetilde{A},\widetilde{g})}$. Another solution is $\widetilde{f}_\xi=f\oplus\xi$, where $\xi\in\cH\setminus\{0\}$ and $\xi\perp\varphi_0$. Clearly, $\widetilde{f}_\xi\notin\overline{\mathcal{K}(\widetilde{A},\widetilde{g})}$.
  \item[(vii)] If $V$ is the Volterra operator on $L^2[0, 1]$ (Sec.~\ref{sec:Volterra}) and $g(x) = \frac{1}{2}x^2$, then $f(x) = x$ is the unique solution to $Vf=g$. On the other hand, $\mathcal{K}(V,g)$ is spanned by the monomials $x^2,x^3,x^4,\dots$, whence 
  \[
  \mathcal{K}(V,g) \;=\; \{x^2p(x)\,|\, p\textrm{ is a polynomial on } [0,1]\}\,. 
  \]
  Therefore $f\notin \mathcal{K}(V,g)$, because $f(x)=x^2\cdot\frac{1}{x}$ and $\frac{1}{x}\notin L^2[0,1]$. Yet, $f\in \overline{\mathcal{K}(V,g)}$, because in fact $\mathcal{K}(V,g)$ is \emph{dense} in $L^2[0,1]$.
  Indeed, if $h\in\mathcal{K}(V,g)^\perp$, then $0=\int_0^1\overline{h(x)}\,x^2p(x)\,\ud x$ for any polynomial $p$; the $L^2$-density of polynomials on $[0,1]$ implies necessarily that $x^2h=0$, whence also $h=0$; this proves that $\mathcal{K}(V,g)^\perp=\{0\}$ and hence $\overline{\mathcal{K}(V,g)}=L^2[0,1]$.
\end{itemize}
\end{example}

\subsection{General conditions for Krylov solvability}~

Even stringent assumptions on $A$ such as the simultaneous occurrence of compactness, normality, injectivity, and density of the range do \emph{not} ensure, in general, that the solution $f$ to $Af=g$, for given $g\in\mathrm{ran}A$, is a Krylov solution (Example \ref{example:Krylov_sol}(v)). 

A \emph{necessary} condition for the solution to a well-defined bounded linear inverse problem to be a Krylov solution, which becomes necessary and sufficient if the linear map is a bounded bijection, is the following. (Recall that for $A\in\mathcal{B}(\cH)$ these three properties are \emph{equivalent}: $A$ is a bijection; $A$ is invertible with everywhere defined, bounded inverse on $\cH$; the spectral point 0 belongs to the resolvent set of $A$.)

\begin{proposition}\label{prop:fKry_iff_AKdenseK}
Let $A$ be a bounded and injective operator on a Hilbert space $\cH$, and let $f\in\cH$ be the solution to $Af=g$, given $g\in\mathrm{ran}A$. One has the following.
 \begin{itemize}
 \item[(i)] If $f\in\overline{\mathcal{K}(A,g)}$, then $A\,\overline{\mathcal{K}(A,g)}$ is dense in $\overline{\mathcal{K}(A,g)}$. 
 \item[(ii)] Assume further that $A$ is invertible with everywhere defined, bounded inverse on $\cH$. Then $f\in\overline{\mathcal{K}(A,g)}$ if and only if $A\,\overline{\mathcal{K}(A,g)}$ is dense in $\overline{\mathcal{K}(A,g)}$. 
 \end{itemize}
\end{proposition}

\begin{proof}
 One has $A\,\overline{\mathcal{K}(A,g)}\supset A\,\mathcal{K}(A,g)=\mathrm{span}\{A^kg\,|\,k\in\mathbb{N}\}$, owing to the definition of Krylov subspace and to \eqref{eq:inclusions_closure}. If $f\in\overline{\mathcal{K}(A,g)}$, then $A\,\overline{\mathcal{K}(A,g)}\ni Af=g$, in which case $A\,\overline{\mathcal{K}(A,g)}\supset \mathrm{span}\{A^kg\,|\,k\in\mathbb{N}_0\}$; the latter inclusion, by means of \eqref{eq:inclusions_closure} and \eqref{eq:invariances2}, implies $\overline{\mathcal{K}(A,g)}\supset\overline{A\,\overline{\mathcal{K}(A,g)}}\supset\overline{\mathcal{K}(A,g)}$, whence
 $\overline{A\,\overline{\mathcal{K}(A,g)}}=\overline{\mathcal{K}(A,g)}$. This proves part (i) and the `only if' implication in part (ii). For the converse, let us now assume that $A^{-1}\in\mathcal{B}(\cH)$ and that $A\,\overline{\mathcal{K}(A,g)}$ is dense in $\overline{\mathcal{K}(A,g)}$. Let $(Av_n)_{n\in\mathbb{N}}$ be a sequence in $A\,\overline{\mathcal{K}(A,g)}$ of approximants of $g\in \overline{\mathcal{K}(A,g)}$ for some $v_n$'s in $\overline{\mathcal{K}(A,g)}$. Since $A^{-1}$ is bounded on $\cH$, then $(v_n)_{n\in\mathbb{N}}$ is a Cauchy sequence, hence $v_n\to v$ as $n\to\infty$ for some $v\in\overline{\mathcal{K}(A,g)}$. By continuity, $Af=g=\lim_{n\to\infty}Av_n=Av$, and by injectivity $f=v\in\overline{\mathcal{K}(A,g)}$, which proves also the `if' implication of part (ii).
\end{proof}

A \emph{sufficient} condition for the Krylov solvability of a well-defined bounded linear inverse problem is the Krylov reducibility introduced in Sec.~\ref{sec:Krylov}.

\begin{proposition}\label{prop:Kry-reduc-implies-Kry-solv}
 Let $A$ be a bounded and injective operator on a Hilbert space $\cH$, and let $f\in\cH$ be the solution to $Af=g$, given $g\in\mathrm{ran}A$. If $A$ is $\mathcal{K}(A,g)$-reduced, then $f\in\overline{\mathcal{K}(A,g)}$. In particular, if $A$ is bounded, injective, and self-adjoint, then $Af=g$ implies $f\in\overline{\mathcal{K}(A,g)}$.
\end{proposition}

\begin{proof}
 Let $P_\mathcal{K}:\cH\to\cH$ be the orthogonal projection onto  $\overline{\mathcal{K}(A,g)}$. On the one hand, $A(\mathbbm{1}-P_\mathcal{K})f\in\overline{\mathcal{K}(A,g)}$, because $AP_\mathcal{K}f\in\overline{\mathcal{K}(A,g)}$ (from \eqref{eq:invariances2} above) and $Af=g\in\overline{\mathcal{K}(A,g)}$. On the other hand, owing to the Krylov reducibility, $A(\mathbbm{1}-P_\mathcal{K})f\in\mathcal{K}(A,g)^\perp$. Then necessarily $A(\mathbbm{1}-P_\mathcal{K})f=0$, and by injectivity $f=P_\mathcal{K}f\in\overline{\mathcal{K}(A,g)}$. 
\end{proof}

In the above proof, Krylov reducibility was only used to deduce that the vector $A(\mathbbm{1}-P_\mathcal{K})f\in A\,\mathcal{K}(A,g)^\perp$ must belong to $\mathcal{K}(A,g)^\perp$; thus, the vanishing of $A(\mathbbm{1}-P_\mathcal{K})f$ -- and hence the same conclusion -- follows also by merely assuming that the Krylov intersection $\overline{\mathcal{K}(A,g)}\cap(A\,\mathcal{K}(A,g)^\perp)$ is trivial. And for bounded bijections, the latter sufficient condition becomes also necessary.

\begin{proposition}\label{prop:KrylIntTriv_KrylSol}
 Let $A$ be a bounded and injective operator on a Hilbert space $\cH$, and let $f\in\cH$ be the solution to $Af=g$, given $g\in\mathrm{ran}A$.
 \begin{itemize}
  \item[(i)] If $\overline{\mathcal{K}(A,g)}\cap(A\,\mathcal{K}(A,g)^\perp)=\{0\}$, then $f\in\overline{\mathcal{K}(A,g)}$.
  \item[(ii)] Assume further that $A$ is invertible with everywhere defined, bounded inverse on $\cH$. Then $f\in\overline{\mathcal{K}(A,g)}$ if and only if $\overline{\mathcal{K}(A,g)}\cap(A\,\mathcal{K}(A,g)^\perp)=\{0\}$.
 \end{itemize}
\end{proposition}

\begin{proof}
 Part (i) and the `if' implication of part (ii) follow as argued right before stating the Proposition. Conversely, if $f\in\overline{\mathcal{K}(A,g)}$ and $A^{-1}\in\mathcal{B}(\cH)$, then $A\,\overline{\mathcal{K}(A,g)}$ is dense in $\overline{\mathcal{K}(A,g)}$ (Proposition \ref{prop:fKry_iff_AKdenseK}(ii)). 
 Let now $z\in\overline{\mathcal{K}(A,g)}\cap(A\,\mathcal{K}(A,g)^\perp)$, say, $z=Aw$ for some $w\in\mathcal{K}(A,g)^\perp$, and based on the density proved above let $(Ax_n)_{n\in\mathbb{N}}$ be a sequence in $A\,\overline{\mathcal{K}(A,g)}$ of approximants of $z$ for some $x_n$'s in $\overline{\mathcal{K}(A,g)}$. From $Ax_n\to z=Aw$ and $\|A^{-1}\|_{\mathrm{op}}<+\infty$ one has $x_n\to w$ as $n\to\infty$. Since $x_n\perp w$, then
 \[
  0\;=\;\lim_{n\to\infty}\|x_n-w\|_{\cH}^2\;=\;\lim_{n\to\infty}\big(\|x_n\|_{\cH}^2+\|w\|_{\cH}^2\big)\;=\;2\|w\|_{\cH}^2\,,
 \]
 whence necessarily $w=0$ and $z=0$. This proves the `only if' implication of (ii).
 \end{proof}

 Propositions \ref{prop:fKry_iff_AKdenseK}(ii) and \ref{prop:KrylIntTriv_KrylSol}(ii) provide equivalent conditions to the Krylov solvability of linear inverse problems on Hilbert space when the linear maps are bounded bijections. (As such, these results do not apply to compact operators on infinite-dimensional Hilbert spaces.)

 In particular, Proposition \ref{prop:KrylIntTriv_KrylSol}(ii) shows that \emph{for such linear inverse problems the Krylov solvability is tantamount as the triviality of the Krylov intersection}, which was the actual reason to introduce the space \eqref{eq:Krilov_intersection}.

 \begin{remark}
  Clearly our interest here is to discuss the possible \emph{occurrence} of Krylov solvability, in the spirit that if the solution $f$ to $Af=g$ is a Krylov solution, then it has the practically favourable feature to be well approximable by linear combinations of $g,Ag,A^2g,\dots$ through one of the many efficient iterative algorithms of the class of the `Krylov subspace methods'. In this case, the next crucial question concerns the \emph{rate of convergence} of the approximate iterate to the actual $f$, a point of view that we do not develop in the present work, but of course is the object of an ample part of the literature -- see, e.g., the monographs \cite{Nevanlinna-1993-converg-iterat-book,Engl-Hanke-Neubauer-1996,Saad-2003_IterativeMethods}.
 \end{remark}

 \subsection{Krylov reducibility and Krylov solvability}~

 Concerning the relation between the Krylov reducibility and the Krylov solvability, we know that the former implies the latter (Prop.~\ref{prop:Kry-reduc-implies-Kry-solv}).

 Moreover, there are classes of operators for which the two notions coincide, as the following remark shows.


\begin{remark}\label{rem:unitaries_KrySolv_KryDecomp}
 For \emph{unitary operators}, the Krylov solvability of the associated inverse problem is \emph{equivalent} to the Krylov-reducibility. Indeed, when $U:\cH\to\cH$ is unitary and $f=U^*g$ is the solution to $Uf=g$ for some $g\in\cH$, then the assumption $f\in\overline{\mathcal{K}(U,g)}$ implies $U^*g\in\overline{\mathcal{K}(U,g)}$, which by Proposition \ref{prop:normops_Kryred_Astarg} is the same as the fact that $U$ is $\mathcal{K}(U,g)$-reduced.
\end{remark}

 There are also Krylov-solvable inverse problems whose operator is not Krylov-reduced, even among well-defined inverse problems, namely when $A$ is bounded and injective and $g\in\mathrm{ran}A$. This is the case of Example \ref{example:Atheta} when $\theta\neq\frac{\pi}{2}$.

 Even in the relevant class of (bounded, injective) \emph{normal} operators (the operator $A_\theta$ of Example \ref{example:Atheta} is \emph{not} normal), Krylov solvability does not necessarily imply Krylov reducibility. Let us discuss a counterexample that elucidates that.

 We need first to recall the following fact from complex functional analysis -- see, e.g., \cite[Theorem 4.4.10]{caruso-thesis-2019}.

 \begin{proposition}\label{prop:closingholo}
  Let $\mathcal{U}\subset\mathbb{C}$ be an open subset of the complex plane, and denote by $H(\mathcal{U})$ the space of holomorphic functions on $\mathcal{U}$. Then the space $H(\mathcal{U})\cap L^2(\mathcal{U})$ is closed in $L^2(\mathcal{U})$.
 \end{proposition}

 \begin{example}
  Consider the multiplication operator $M_z:L^2(\Omega)\to L^2(\Omega)$, $M_zf= zf$, with $\Omega=\{z\in\mathbb{C}\,|\,|z-2|<1\}$, and let $g\in L^2(\Omega)$ be such that $\varepsilon\leqslant |g(z)|\leqslant \varepsilon'$ $\forall z\in\Omega$, for given $\varepsilon,\varepsilon'>0$. Then:
  \begin{itemize}
   \item[(i)] $M_z$ is bounded, injective, and normal;
   \item[(ii)] the inverse linear problem $M_zf=g$ is Krylov-solvable: $f\in\overline{K(M_z,g)}$;
   \item[(iii)] however, $M_z$ is not Krylov-reduced.
  \end{itemize}
 Parts (i) and (ii) were discussed in Example \ref{example:Krylov_sol}(ii). It was also observed therein that  $\mathcal{K}(M_z,g)=\{p\,g\,|\,p\in\mathbb{P}_\Omega[z]\}$, where $\mathbb{P}_\Omega[z]$ denotes the polynomials in $z\in\Omega$ with complex coefficients. Let us show that
 \[\tag{*}
  \overline{\mathcal{K}(M_z,g)}\;=\;\Big\{\phi g\,\Big|\,\phi\in\overline{\mathbb{P}_\Omega[z]\,}^{\|\,\|_2}\Big\}\,.
 \]
 Indeed, if $w\in\overline{\mathcal{K}(M_z,g)}$, then $w\xleftarrow[]{\|\,\|_2} p_n g$ for a sequence $(p_n)_{n\in\mathbb{N}}$ in $\mathbb{P}_\Omega[z]$, and since 
 \[
  \|p_n-p_m\|_{L^2(\Omega)}\;\leqslant\;\varepsilon^{-1}\|gp_n-gp_m\|_{L^2(\Omega)}
 \]
 then $(p_n)_{n\in\mathbb{N}}$ is a Cauchy sequence in $L^2(\Omega)$ with $p_n\xrightarrow[]{\|\,\|_2}\phi$ for some $\phi\in\overline{\mathbb{P}_\Omega[z]\,}^{\|\,\|_2}$; whence $w=\phi g$. Conversely, if $w=\phi g$ for $\phi\in\overline{\mathbb{P}_\Omega[z]\,}^{\|\,\|_2}$, then $\phi\xleftarrow[]{\|\,\|_2} p_n$ for a sequence $(p_n)_{n\in\mathbb{N}}$ of approximants in $\mathbb{P}_\Omega[z]$ and 
 \[
  \|w-p_ng\|_{L^2(\Omega)}\;=\;\|\phi g-p_n g\|_{L^2(\Omega)}\;\leqslant\;\varepsilon'\|\phi-p_n\|_{L^2(\Omega)}\;\xrightarrow[]{n\to\infty}\;0
 \]
 shows that $w\in\overline{\mathcal{K}(M_z,g)}$. The identity (*) is therefore established. Now, if by contradiction $M_z$ was reduced with respect to the decomposition $L^2(\Omega)=\overline{\mathcal{K}(M_z,g)}\oplus \mathcal{K}(M_z,g)^\perp$, then $\overline{z}g=M_{\overline{z}}g=M_z^*g\in\overline{\mathcal{K}(M_z,g)}$ (Prop.~\ref{prop:normops_Kryred_Astarg}), and identity (*) would imply that the function $\Omega\ni z\mapsto\overline{z}$ belongs to $\overline{\mathbb{P}_\Omega[z]\,}^{\|\,\|_2}$; however, the latter space, owing to Prop.~\ref{prop:closingholo}, is formed by holomorphic functions, and the function $z\mapsto\overline{z}$ clearly is not. Part (iii) is thus proved. 
 \end{example}

 \subsection{More on Krylov solutions in the lack of well-posedness}~

 Let us consider more generally solvable inverse problem ($g\in\mathrm{ran}A$) which are not necessarily well-posed (i.e., $A$ is possibly non-injective). First, we see that Krylov reducibility still guarantees the existence of Krylov solutions, indeed Prop.~\ref{prop:Kry-reduc-implies-Kry-solv} has a counterpart valid also in the lack of injectivity, which reads as follows.

 \begin{proposition}\label{prop:Kry-reduc-implies-Kry-solv-2}
 Let $A$ be a bounded operator on a Hilbert space $\cH$, and let $g\in\mathrm{ran}A$. If $A$ is $\mathcal{K}(A,g)$-reduced, then there exists a Krylov solution to the problem $Af=g$. For example, if $f_\circ\in\cH$ satisfies $Af_\circ=g$ and $P_{\mathcal{K}}$ is the orthogonal projection onto $\overline{\mathcal{K}(A,g)}$, then $f:=P_{\mathcal{K}}f_\circ$ is a Krylov solution.
\end{proposition}

\begin{proof}
 One has $A(\mathbbm{1}-P_\mathcal{K})f_\circ=0$, owing to the very same argument as in the above proof of Prop.~\ref{prop:Kry-reduc-implies-Kry-solv}. Thus, $AP_\mathcal{K}f_\circ=g$, that is, $f:=P_{\mathcal{K}}f_\circ$ is a Krylov solution.
\end{proof}
 
 Generic bounded linear inverse problems may or may not admit a Krylov solution, and when they do there may exist further non-Krylov solutions (Example \ref{example:Krylov_sol}). For a fairly general class of such problems, however, the Krylov solution, when it exists, is \emph{unique}.

 \begin{proposition}\label{prop:UniqueKrySol_normalOps}
   Let $A$ be a bounded normal operator on a Hilbert space $\cH$, and let $Af=g$ be the associated linear inverse problem, given $g\in\mathrm{ran}A$. Then there exists at most one solution $f\in\overline{\mathcal{K}(A,g)}$. More generally, the same conclusion holds if $A$ is bounded with $\ker A\subset\ker A^*$. 
 \end{proposition}

 \begin{proof}
  If $f_1,f_2\in\overline{\mathcal{K}(A,g)}$ and $Af_1=g=Af_2$, then $f_1-f_2\in\ker A\cap \overline{\mathcal{K}(A,g)}$. By normality, $\ker A=\ker A^*$, and  moreover obviously $\overline{\mathcal{K}(A,g)}\subset\overline{\mathrm{ran}A}$. Therefore, $f_1-f_2\in\ker A^*\cap \overline{\mathrm{ran}A}$. But $\ker A^*\cap \overline{\mathrm{ran}A}=\{0\}$, whence $f_1=f_2$. The second statement is then obvious.
 \end{proof}
This proposition is similar to comments made in \cite{Freund-Hochbruck-1994,Brown-Walker-1997,Gasparo-Papini-Pasquali-2008} about Krylov solutions to singular systems in finite dimensions.

 Propositions \ref{prop:Kry-reduc-implies-Kry-solv-2} and \ref{prop:UniqueKrySol_normalOps} above have a noticeable consequence.

\begin{corollary}\label{cor:self-adj_Kry}
 If $A\in\mathcal{B}(\cH)$ is self-adjoint, then the inverse problem $Af=g$ with $g\in\mathrm{ran}A$ admits a unique Krylov solution.
\end{corollary}

\begin{proof}
 $A$ is $\mathcal{K}(A,g)$-reduced (Example \eqref{example:Kry-red}(i)), hence the induced inverse problem admits a Krylov solution (Prop.~\ref{prop:Kry-reduc-implies-Kry-solv-2}). Such a solution is then necessarily unique (Prop.~\ref{prop:UniqueKrySol_normalOps}).
\end{proof}

It is worth noticing that the self-adjoint case has always deserved a special status in this context, theoretically and in applications: the convergence of Krylov techniques for self-adjoint operators are the object of an ample literature -- see, e.g., \cite{Karush-1952,Daniel-1967,Kammerer-Nashed-1972,Winther-1980,Nemirovskiy-Polyak-1985,Herzog-Ekkehard-2015,Novati-2018-KryTikh}.

\begin{example}\label{example:HansenRegTools}
The test problems
\[
 \begin{array}{lllll}
  \textsf{blur} & \textsf{deriv2} & \textsf{foxgood} & \textsf{gravity} & \textsf{heat} \\
    \textsf{i\_laplace} & \textsf{parallax} & \textsf{phillips} & \textsf{shaw} & \textsf{ursell}
 \end{array}
\]
of Hansen's REGULARIZATION TOOLS Matlab package \cite{Hansen-RegToolbox_v4-1} correspond to integral operators $A_K$ on some $L^2[a,b]$ whose integral kernels $K(x,y)$ are square-integrable and have the property $K(x,y)=\overline{K(x,y)}$, namely they are Hilbert-Schmidt and self-adjoint operators. Owing to Corollary \ref{cor:self-adj_Kry}, all such inverse problems admit a unique Krylov solution (in fact they are Krylov-solvable, as long as $g\in\mathrm{ran}A_K$ and $A_K$ is injective).
\end{example}

\begin{example}\label{example:PRdiffusion} 
 The {\tt PRdiffusion} two-dimensional test problem of Gazzola-Hansen-Nagy's IR Tools \cite{Gazzola-Hansen-Nagy-2018} consists of reconstructing, from the heat diffusion problem
 \[
  \begin{cases}
   \;\;\,\displaystyle\frac{\partial u}{\partial t}=\Delta_N u \\
   \:u(0) = u_0
  \end{cases}
 \]
 with unknown $u\equiv u(x,y;t)$ in the Hilbert space $L^1([0,1]\times[0,1])$ and with Neumann Laplacian $\Delta_N$, the initial datum $u_0$ starting from the knowledge of the function $u(t)$ at time $t>0$. By standard functional-analytic arguments one has $u_0=e^{-t\Delta_N}u(t)$, that is, for given $t$ the inverse problem $u(t)\mapsto u_0$ is self-adjoint and hence (Corollary \ref{cor:self-adj_Kry}) it admits a unique Krylov solution.
\end{example}

\begin{example}\label{example:nonselfadj}
For given $k\in L^2[0,1]$, it is a standard fact that the operator
\begin{equation}
 \begin{split}
  & A:L^2[0,1]\to L^2[0,1] \\
  &  (A u)(x)\;:=\int_0^1 k(x-y) u(y)\,\ud y
 \end{split}
\end{equation}
is a Hilbert-Schmidt normal operator, with norm $\|A\|_{\mathrm{op}}\leqslant\|k\|_{L^2}$ (the integral kernel $\kappa_{A^*A}$ of $A^*A$ being the function $\kappa_{A^*A}(x,y)=\int_0^1\overline{k(y-z)}\,k(x-z)\,\ud z$). $A$ is self-adjoint if and only if $k(x)=k(-x)$ for almost every $x\in[0,1]$: when the latter condition is not matched, Corollary \ref{cor:self-adj_Kry} is not applicable any longer. This is the case if we consider, for concreteness, the function
\begin{equation}
 k(x)\;:=\;\frac{e\,\sin\pi x\,}{\,(1+e)\pi\,e^x}-\frac{1}{\,1+\pi^2\,}\,.
\end{equation}
With the above choice, in order investigate the Krylov-solvability of the inverse problem $Af=g$ for given $g\in\mathrm{ran}\,A$ one must then go through an ad hoc analysis. Let us introduce the orthonormal basis $\{\varphi_n\,|\,n\in\mathbb{Z}\}$ of $L^2[0,1]$, where
\begin{equation}
 \varphi_n(x)\;=\; e^{2\pi\ii n x}\,.
\end{equation}
Then 
\begin{equation}
 k\;=\;\sum_{n\in\mathbb{Z}}c_n\,\varphi_n\,,\qquad c_n\,:=\,\langle \varphi_n,k\rangle_{L^2}\,,
\end{equation}
and a straightforward explicit computation yields
\begin{equation}
 c_n\;=\;
 \begin{cases}
  \;0 & \textrm{if }\,n=0 \\
  \displaystyle\frac{1}{\,1+4\ii n\pi+(1-4n^2)\pi^2\,} & \textrm{if }\,n\in\mathbb{Z}\setminus\{0\}\,.
 \end{cases}
\end{equation}
As a consequence,
\begin{equation}
 \begin{split}
  (A u)(x)\;&=\int_0^1 k(x-y) u(y)\,\ud y\;=\;\sum_{n\in\mathbb{Z}}c_n\int_0^1 \varphi_n(x-y) u(y)\,\ud y \\
  &=\sum_{n\in\mathbb{Z}}c_n\,\varphi_n(x)\int_0^1 \overline{\varphi_n(y)}\, u(y)\,\ud y\,,
 \end{split}
\end{equation}
that is,
\begin{equation}\label{eq:Aexpanded}
 A\;=\;\sum_{n\in\mathbb{Z}}c_n|\varphi_n\rangle\langle \varphi_n|\;=\;\sum_{n\in\mathbb{Z}}\lambda_n|\psi_n\rangle\langle \varphi_n|\,,\qquad 
 \begin{cases}
  \lambda_n\,:=\,|c_n|\in\,\mathbb{R} \\
  \psi_n\,:=\,e^{\ii\,\mathrm{arg}(c_n)}\varphi_n\,.
 \end{cases}
\end{equation}
It is standard to see that $\{\psi_n\,|\,n\in\mathbb{Z}\}$ is just another orthonormal basis of $L^2[0,1]$ and that the convergence in \eqref{eq:Aexpanded} holds in the operator norm. Thus, the second equality in \eqref{eq:Aexpanded} gives the usual singular value decomposition of $A$. We can now draw a number of conclusions.
\begin{itemize}
 \item $A$ is not injective: $\ker A=\mathrm{span}\{\varphi_0\}$.
 \item $\mathrm{ran}A=\mathrm{span}\{\psi_n\,|\,n\in\mathbb{Z}\setminus\{0\}\}=\mathrm{span}\{\varphi_n\,|\,n\in\mathbb{Z}\setminus\{0\}\}$.
 \item If $g\in\mathrm{ran}A$ (that is, if $g$ is not a constant function), and $J\subset\mathbb{Z}\setminus\{0\}$ is the subset of non-zero integers $n$ such that $g_n:=\langle \psi_n,g\rangle_{L^2}\neq 0$, then
 \[
  g\;=\;\sum_{n\in J}\;g_n \,\psi_n
 \]
 and the inverse linear problem $Af=g$ admits an infinity of solutions of the form $f=\alpha \varphi_0+f_K$ for arbitrary $\alpha\in\mathbb{C}$, where 
 \[
  f_K\;:=\;\sum_{n\in J}\;\frac{g_n}{\lambda_n} \,\varphi_n
 \]
  (recall that $\lambda_n\neq 0$ whenever $n\neq 0$).
 \item Moreover, due to the property $\psi_n=e^{\ii\,\mathrm{arg}(c_n)}\varphi_n$ the vectors $g,Ag,A^2g,\dots$ have non-zero components only of order $n\in J$; this, together with the fact that the $\lambda_n$'s are all distinct, implies that
 \[
  K(A,g)\;=\;\mathrm{span}\{\psi_n\,|\,n\in J\}=\mathrm{span}\{\varphi_n\,|\,n\in J\}\,.
 \]
 \item The functions $\alpha\varphi_0$ with $\alpha\in\mathbb{R}\setminus\{0\}$ are non-Krylov solutions to the problem $Af=g$, whereas $f_K$ is the unique Krylov solution, consistently with Prop.~\ref{prop:UniqueKrySol_normalOps}.
\end{itemize}
\end{example}

 \subsection{Special classes of Krylov-solvable problems}~

 In the current lack (to our knowledge) of a complete characterisation of all Krylov-solvable inverse problems on infinite-dimensional Hilbert space, it is of interest to identify special sub-classes of them.

 We already examined simple explicit cases in Example \ref{example:Krylov_sol}, parts (i), (ii), (iii), and (vii).

 We also concluded (Cor.~\ref{cor:self-adj_Kry}) that a whole class of paramount relevance, the bounded self-adjoint operators, induce inverse problems that are Krylov-solvable, and Examples \ref{example:HansenRegTools} and \ref{example:PRdiffusion} survey a number of applications.

 It is worth mentioning that in the special case where the operator $A$ is  bounded, self-adjoint, and \emph{positive definite}, an alternative analysis by Nemirovskiy and Polyak \cite{Nemirovskiy-Polyak-1985} (for a more recent discussion of which we refer to \cite[Sect.~7.2]{Engl-Hanke-Neubauer-1996} and \cite[Sect.~3.2]{Hanke-ConjGrad-1995}, as well as \cite{CM-Nemi-unbdd-2019}) proved that the corresponding linear inverse problem $Af=g$ with $g\in\ran A$ is actually Krylov-solvable. In particular it was proved that the sequence of Krylov approximations from the conjugate gradient algorithm converges strongly to the exact solution.

 Krylov-solvable problems can be surely found for suitable \emph{non}-self-adjoint operators too (Example \ref{example:nonselfadj}), although, as already commented, Krylov-solvability is not automatic for compact, normal, injective operators (Example \ref{example:Krylov_sol}(v)).

 To conclude this Section, let us present one further class of well-posed inverse linear problems that are Krylov-solvable (Corollary \ref{cor:KclassSolvable} below). For shortness, we shall say that an operator $A$ is of \emph{class}-$\mathscr{K}$ when
 \begin{itemize}
  \item $A\in\mathcal{B}(\cH)$,
  \item $0\notin\sigma(A)$,
  \item there exists an open subset $\mathcal{W}\subset\mathbb{C}$ such that $\sigma(A)\subset\mathcal{W}$, $\overline{\mathcal{W}}$ is compact with $0\notin\overline{\mathcal{W}}$, and $\mathbb{C}\setminus\overline{\mathcal{W}}$ is connected in $\mathbb{C}$.
 \end{itemize}
 (Observe, for instance, that the multiplication operator $M_z$ considered in Example \ref{example:Krylov_sol}(ii) is of class-$\mathscr{K}$, whereas unitary operators are not.)

  Class-$\mathscr{K}$ operators have a polynomial approximation of their inverse, which eventually yields Krylov-solvability of the associated inverse problem.

 \begin{proposition}\label{prop:inverseclassk}
  Let $A$ be an operator of class-$\mathscr{K}$ on a Hilbert space $\cH$. Then there exists a polynomial sequence $(p_n)_{n\in\mathbb{N}}$ over $\mathbb{C}$ such that $\|p_n(A)-A^{-1}\|_{\mathrm{op}}\to 0$ as $n\to\infty$.  
 \end{proposition}

 \begin{proof}
  Let $\mathcal{U}\subset\mathbb{C}$ be an open set such that $0\notin\mathcal{U}$ and $\overline{\mathcal{W}}\subset\mathcal{U}$, where $\mathcal{W}$ is an open set fulfilling the definition of class-$\mathscr{K}$ for the given $A$. The function $z\mapsto z^{-1}$ is holomorphic on $\mathcal{U}$ and hence (see, e.g., \cite[Theorem 13.7]{Rudin-realcomplexanalysis}) there exists a polynomial sequence $(p_n)_{n\in\mathbb{Z}}$ on $\mathcal{W}$ such that
  \[
   \|z^{-1}-p_n(z)\|_{L^\infty(\overline{\mathcal{W}})}\;\xrightarrow[]{n\to\infty}\;0\,.
  \]
  On the other hand, there exists a closed curve $\Gamma\subset\mathcal{W}\setminus\sigma(A)$ such that (see, e.g., \cite[Theorem 13.5]{Rudin-realcomplexanalysis}) 
  \[
    z^{-1}\;=\;\frac{1}{2\pi\ii}\int_{\Gamma}\frac{\ud\zeta}{\,\zeta(\zeta-z)}\,,\qquad p_n(z)\;=\;\frac{1}{2\pi\ii}\int_{\Gamma}\frac{p_n(\zeta)}{(\zeta-z)}\,\ud\zeta\,,
  \]
 whence also (see, e.g. \cite[Chapter XI, Sect.~151]{Riesz-Nagy_FA-1955_Eng})
 \[
    A^{-1}\;=\;\frac{1}{2\pi\ii}\int_{\Gamma}\zeta^{-1}(\zeta\mathbbm{1}-A)^{-1}\,\ud\zeta\,,\qquad p_n(A)\;=\;\frac{1}{2\pi\ii}\int_{\Gamma}p_n(\zeta)\,(\zeta\mathbbm{1}-A)^{-1}\,\ud\zeta\,.
  \]
  Thus,
  \[
   \begin{split}
    \|A^{-1}-p_n(A)\|_{\mathrm{op}}\;&=\;\Big\|\frac{1}{\,2\pi\ii\,}\int_{\Gamma}(\zeta^{-1}-p_n(\zeta))(\zeta\mathbbm{1}-A)^{-1}\,\ud\zeta \Big\|_{\mathrm{op}} \\
    &\leqslant\;\|z^{-1}-p_n(z)\|_{L^\infty(\overline{\mathcal{W}})}\;\Big\|\frac{1}{\,2\pi\ii\,}\int_{\Gamma}(\zeta\mathbbm{1}-A)^{-1}\,\ud\zeta \Big\|_{\mathrm{op}} \\
    &=\;\|z^{-1}-p_n(z)\|_{L^\infty(\overline{\mathcal{W}})}
   \end{split}
  \]
  (indeed, $(2\pi\ii)^{-1}\int_{\Gamma}(\zeta\mathbbm{1}-A)^{-1}\,\ud\zeta=\mathbbm{1}$), and the conclusion follows.
 \end{proof}

 \begin{corollary}\label{cor:KclassSolvable}
  Let $A$ be an operator of class-$\mathscr{K}$ on a Hilbert space $\cH$. Then the inverse problem $Af=g$ for given $g\in\cH$ is Krylov-solvable, i.e., the unique solution $f$ belongs to $\overline{\mathcal{K}(A,g)}$.  
 \end{corollary}

 \begin{proof}
  As $\|p_n(A)-A^{-1}\|_{\mathrm{op}}\xrightarrow[]{n\to\infty} 0$ (Prop.~\ref{prop:inverseclassk}), then also $\|p_n(A)g-f\|_{\cH}=\|p_n(A)g-A^{-1}g\|_{\cH}\xrightarrow[]{n\to\infty} 0$, and obviously $p_N(A)g\in\mathcal{K}(A,g)$.  
 \end{proof}

\section{Numerical tests and examples}\label{sec:numerics}

In this final Section we examine the main features discussed theoretically so far through a series of numerical tests on inverse problems in infinite-dimensional Hilbert space, suitably truncated using the GMRES algorithm, and analysed by increasing the size of the truncation (i.e. the number of iterations of GMRES).

We focus on the behaviour of the truncated problems under these circumstances:
\begin{itemize}
 \item[I)] when the solution to the original problem is or is not a Krylov solution;
 \item[II)] when the linear operator is or is not injective (well-defined vs ill-defined problem).
\end{itemize}

\subsection{Four inverse linear problems}~

As a `baseline' case, where the solution is known a priori to be a Krylov solution, we considered the compact, injective, self-adjoint multiplication operator on $\ell^2(\mathbb{N})$ (Sec.~\ref{sec:multiplication_op})
\begin{equation}
  M\;=\;\sum_{n=1}^\infty \sigma_n |e_{n}\rangle\langle e_n|\,,\qquad \sigma_n\,=\,(5n)^{-1}, \label{eq:diagonal}
\end{equation}
In comparison to $M$ we tested a non-injective version of it, namely
\begin{equation}\label{eq:diagonal-noninjective}
  \widetilde{M}\;=\;\sum_{n=1}^\infty \widetilde{\sigma}_n |e_{n}\rangle\langle e_n|\,,\qquad
  \widetilde{\sigma}_n\,=\,
  \begin{cases}
   0 & \textrm{ if } n\in\{3,6,9\} \\
   \sigma_n & \textrm{ otherwise},
  \end{cases}
\end{equation}
as well as the weighted right shift (Sec.~\ref{sec:compactRshift})
\begin{equation}
 \mathcal{R}\;=\;\sum_{n=1}^\infty \sigma_n |e_{n+1}\rangle\langle e_n|
\end{equation}
with the same weights as in \eqref{eq:diagonal}.
We thus investigated the inverse problems $Mf=g$, $\widetilde{M}f=g$, and $\mathcal{R}f=g$ with datum $g$ generated by the a priori chosen solution
\begin{equation}
 f\;=\;\sum_{n\in\mathbb{N}}f_n e_n\,,\qquad f_n\,=\,  
 	\begin{cases}
   n^{-1} & \textrm{ if } n\leqslant 250 \\
   \;0 & \textrm{ otherwise}\,.
  \end{cases}
\end{equation}
Let us observe that 
\begin{equation}\label{eq:digamma}
 \|f\|_{\ell^2}\;=\; \sqrt{{\textstyle\frac{\;\pi^2}{6}} - \Psi^{(1)}(251)}\;\simeq\; 1.28099\,,
\end{equation}
where $\Psi^{(k)}$ is the polygamma function of order $k$ \cite[Sec.~6.4]{Abramowitz-Stegun-1964}.

Fourth and last, we considered the inverse problem $Vf=g$ where $V$ is the Volterra operator in $L^2[0, 1]$ (Sect.~\ref{sec:Volterra}) and $g(x)=\frac{1}{2}x^2$. The problem has unique solution
 \begin{equation}\label{eq:solV}
  f(x)\;=\;x\,,\qquad \|f\|_{L^2[0,1]}\;=\;\frac{1}{\sqrt{3}}\;\simeq\;0.5774\,.
 \end{equation}

Depending on the context, we shall denote respectively by $\cH$ and by $A$ the Hilbert space ($\ell^2(\mathbb{N})$ or $L^2[0,1]$) and the operator ($M$, $\widetilde{M}$, $\mathcal{R}$, or $V$) under consideration.

The inverse problems in $\cH$ associated with $M$ and $\widetilde{M}$ are Krylov-solvable (Corollary \ref{cor:self-adj_Kry}), and so too is the inverse problem associated with $V$, with $\mathcal{K}(V,g)$ dense in $L^2[0,1]$ (Example \ref{example:Krylov_sol}(vii)).

Instead, the problem associated with $\mathcal{R}$ is not Krylov-solvable, for $\mathcal{K}(\mathcal{R},g)^\perp$ always contains the first canonical vector $e_1$.

For each operator $A$, we proceeded numerically by generating the spanning vectors $g,Ag,A^2g,\dots$ of $\mathcal{K}(A,g)$ up to order $N_{\mathrm{max}}=500$ if $A=M,\widetilde{M},\mathcal{R}$, and up to order $N_{\mathrm{max}}=175$ if $A=V$. Such values represent our practical choice of `infinite' dimension for $\mathcal{K}(A,g)$.

Analogously, when $A=M,\widetilde{M},\mathcal{R}$ we allocated for each of the considered vectors $f,g,Ag,A^2g,\dots$ an amount of 2500 entries with respect to the canonical basis of $\ell^2(\mathbb{N})$: such a value represents our practical choice of `infinite' dimension for $\cH$. Let us observe, in particular, that by repeated application of $\mathcal{R}$ up to 500 times, the vectors $\mathcal{R}^k g$ have non-trivial entries up to order 251+500=751 (by construction the last non-zero entries of $f$ and of $g$ are the components, respectively, $e_{250}$ and $e_{251}$), and by repeated application of $M$ and $\widetilde{M}$ the vectors $M^kg$ and $\widetilde{M}^kg$ have the component $e_{250}$ as last non-zero entry: all such limits stay well below our `infinity' threshold of 2500 for $\cH$.

From each collection $\{g,Ag,\dots,A^{N-1}g\}$ we then obtained an orthonormal basis of the $N$-dimensional truncation of $\mathcal{K}(A,g)$, $N\leqslant N_{\mathrm{max}}$, and we truncated the `infinite-dimensional' inverse problem $Af=g$ to a $N$-dimensional one, that we solved by means of the GMRES algorithm, in the same spirit of our general discussion \cite[Sect.~2]{CMN-truncation-2018}.

Denoting by  $\widehat{f^{(N)}}\in\cH$ the vector of the solution from the GMRES algorithm at the $N$-th iterate, 
we analysed two natural indicators of the convergence `as $N\to\infty$', the \emph{infinite-dim\-en\-sion\-al error} $\mathscr{E}_N$ and the \emph{infinite-dim\-en\-sion\-al residual} $\mathfrak{R}_N$, defined respectively \cite[Sect.~2]{CMN-truncation-2018} as 
\begin{equation}\label{eq:error-residual}
   \mathscr{E}_N\;:=\;f-\widehat{f^{(N)}}\,,\qquad
   \mathfrak{R}_N\;:=\;g-A\,\widehat{f^{(N)}}\,.
\end{equation}

\subsection{Krylov vs non-Krylov solutions}~

The (norm) behaviours of the infinite-dimensional error $\|\mathscr{E}_N\|_{\mathcal{H}}$, of the infinite-dimensional residual $\| \mathfrak{R}_N \|_{\mathcal{H}}$, and of the approximated solution $\Vert \widehat{f^{(N)}} \Vert_{\mathcal{H}}$ at the $N$-th step of the algorithm are illustrated in Figure~\ref{fig:error4cases} as a function of $N$.

The numerical evidence is the following.
\begin{itemize}
 \item The error norm of the baseline case and the Volterra case tend to vanish with $N$, and so does the residual norm, consistently with the obvious property $\| \mathfrak{R}_N \|_{\mathcal{H}}\leqslant\|A\|_{\mathrm{op}}\|\mathscr{E}_N\|_{\mathcal{H}}$. Moreover, $\Vert \widehat{f^{(N)}} \Vert_{\mathcal{H}}$ stays uniformly bounded and attains asymptotically the theoretical value prescribed by \eqref{eq:digamma} or \eqref{eq:solV}.
 \item Instead, the error norm of the forward shift remains of order one indicating a lack of \emph{norm-convergence}, regardless of truncation size. Analogous lack of convergence is displayed in the norm of the finite-di\-men\-sional residual. Again, $\Vert \widehat{f^{(N)}} \Vert_{\mathcal{H}}$ remains uniformly bounded, but attains an asymptotic value that is strictly smaller than the theoretical value \eqref{eq:digamma}.
\end{itemize}

\begin{figure}[!t]
  \centering
  \begin{subfigure}[b]{\textwidth}
    \includegraphics[width = 0.32 \textwidth]{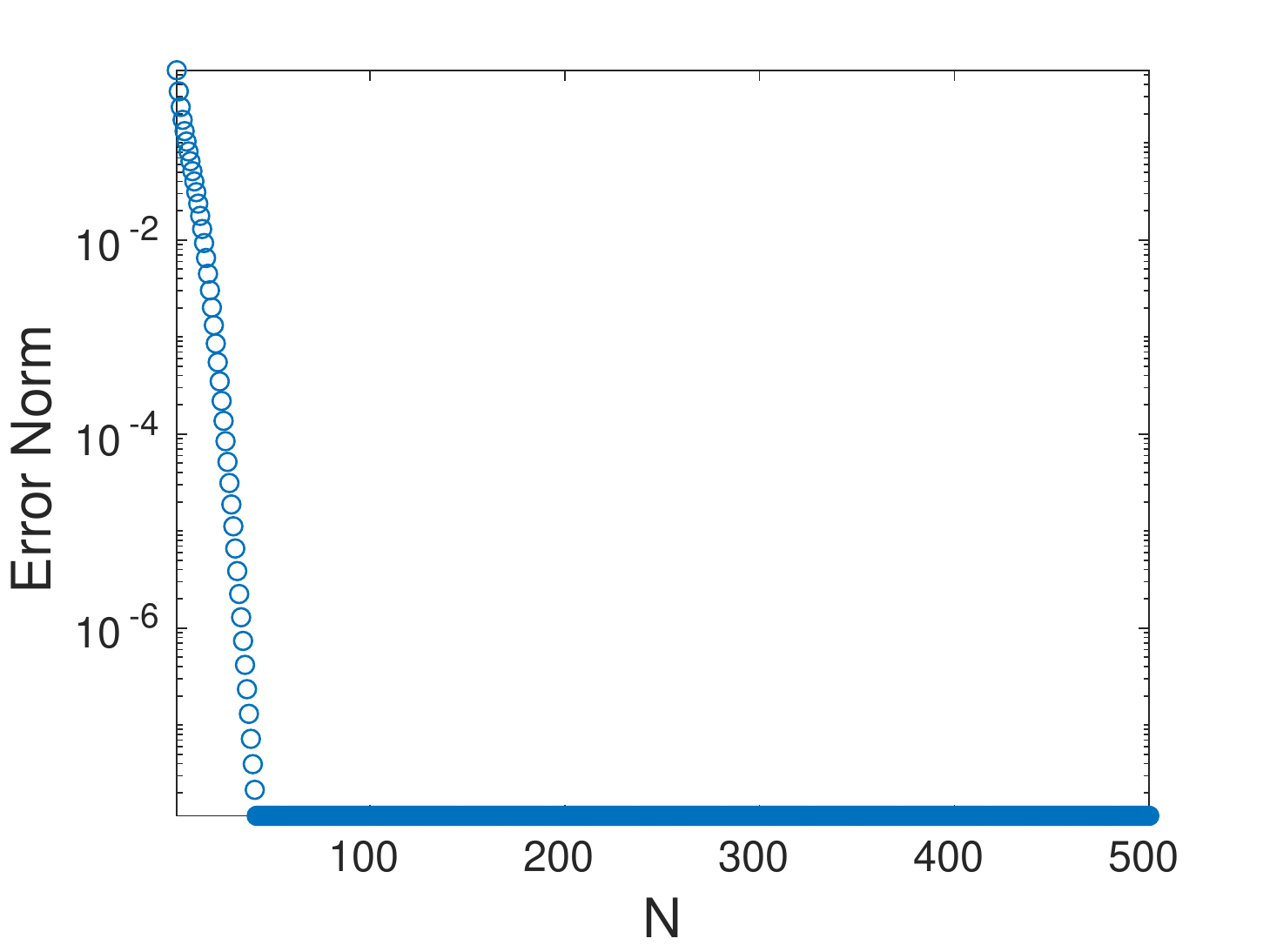}
    \includegraphics[width = 0.32 \textwidth]{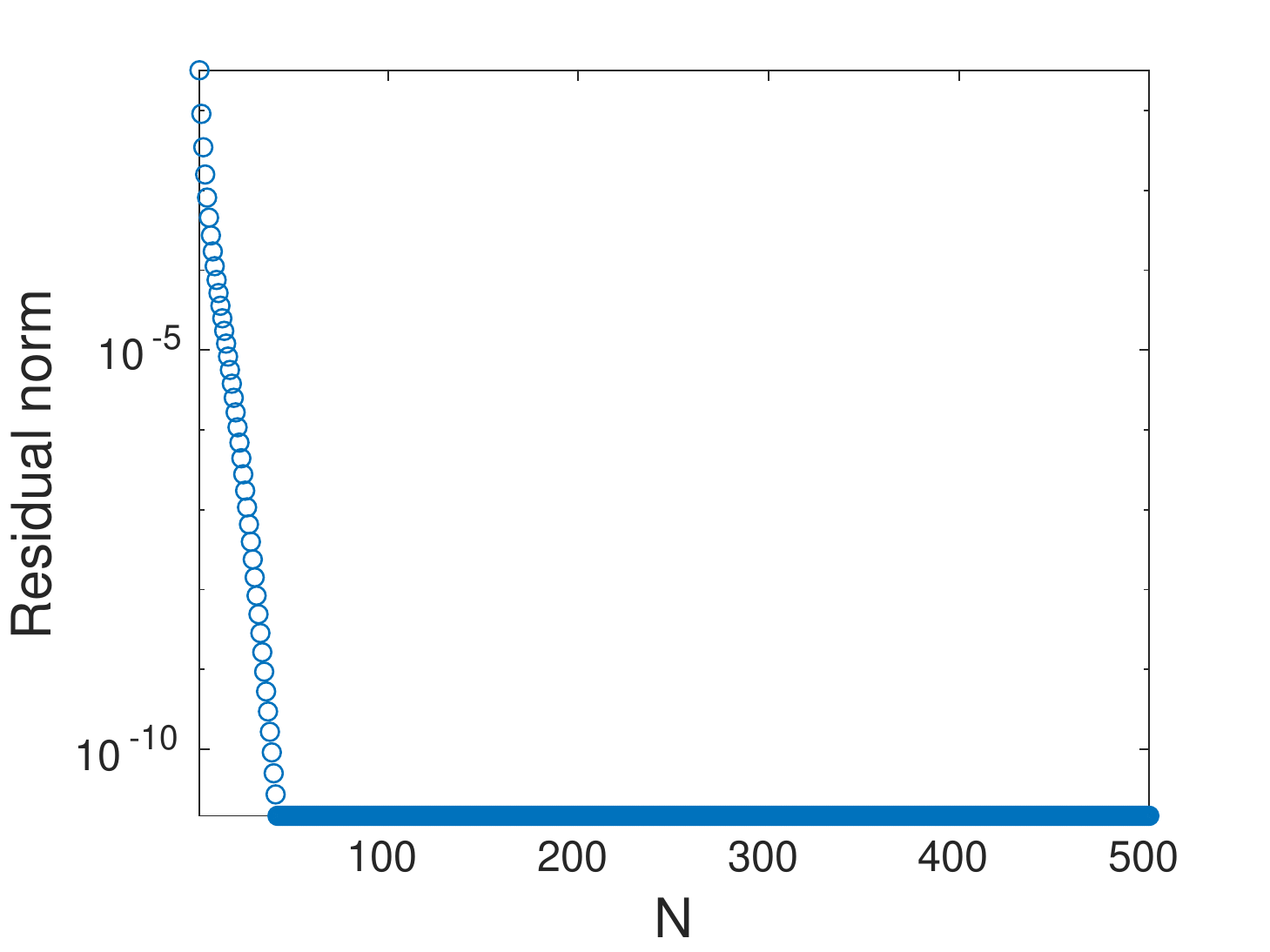}
    \includegraphics[width = 0.32 \textwidth]{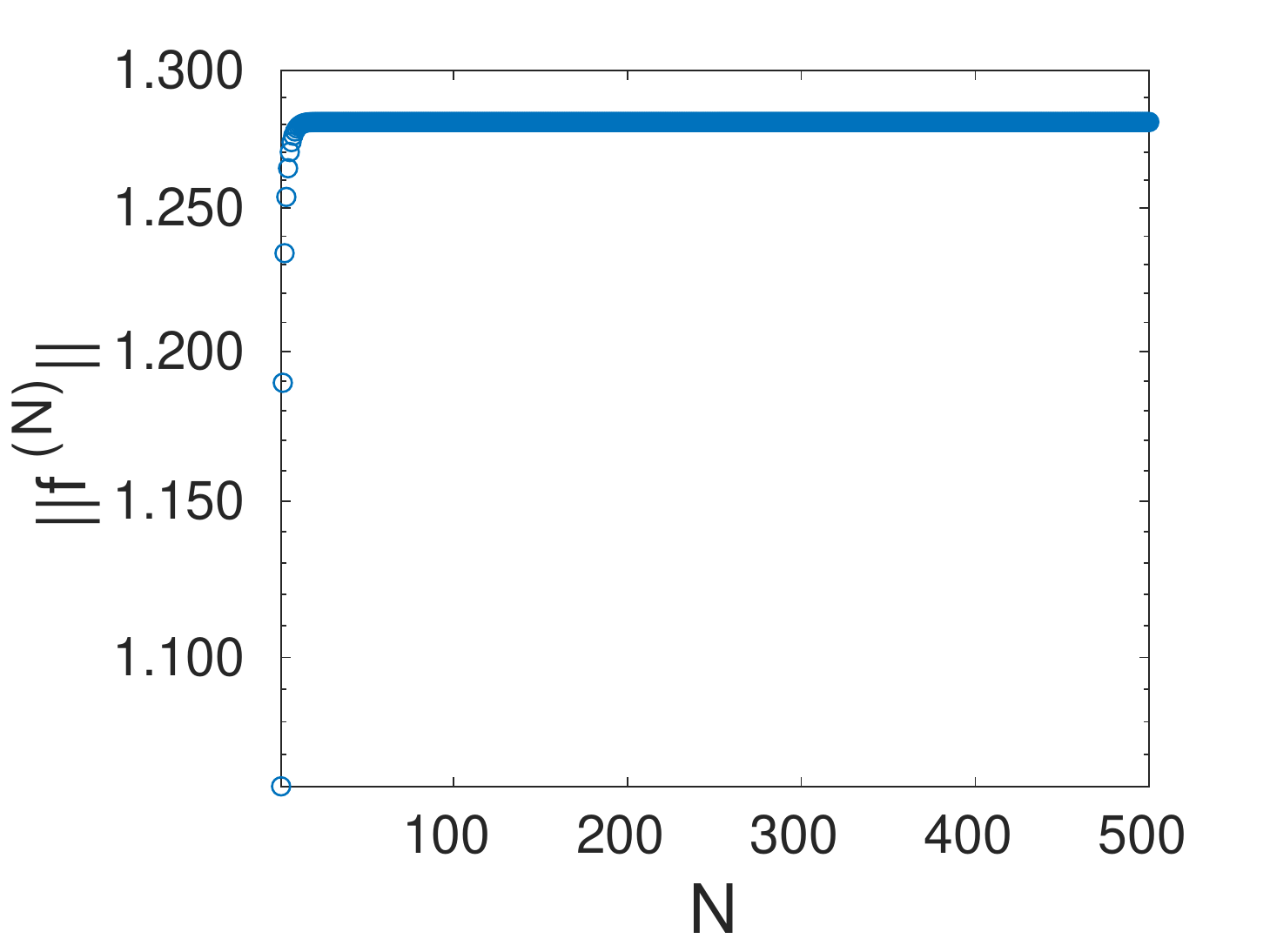}
    \caption{Case $M$}
  \end{subfigure}
  \begin{subfigure}[b]{\textwidth}
    \includegraphics[width = 0.32 \textwidth]{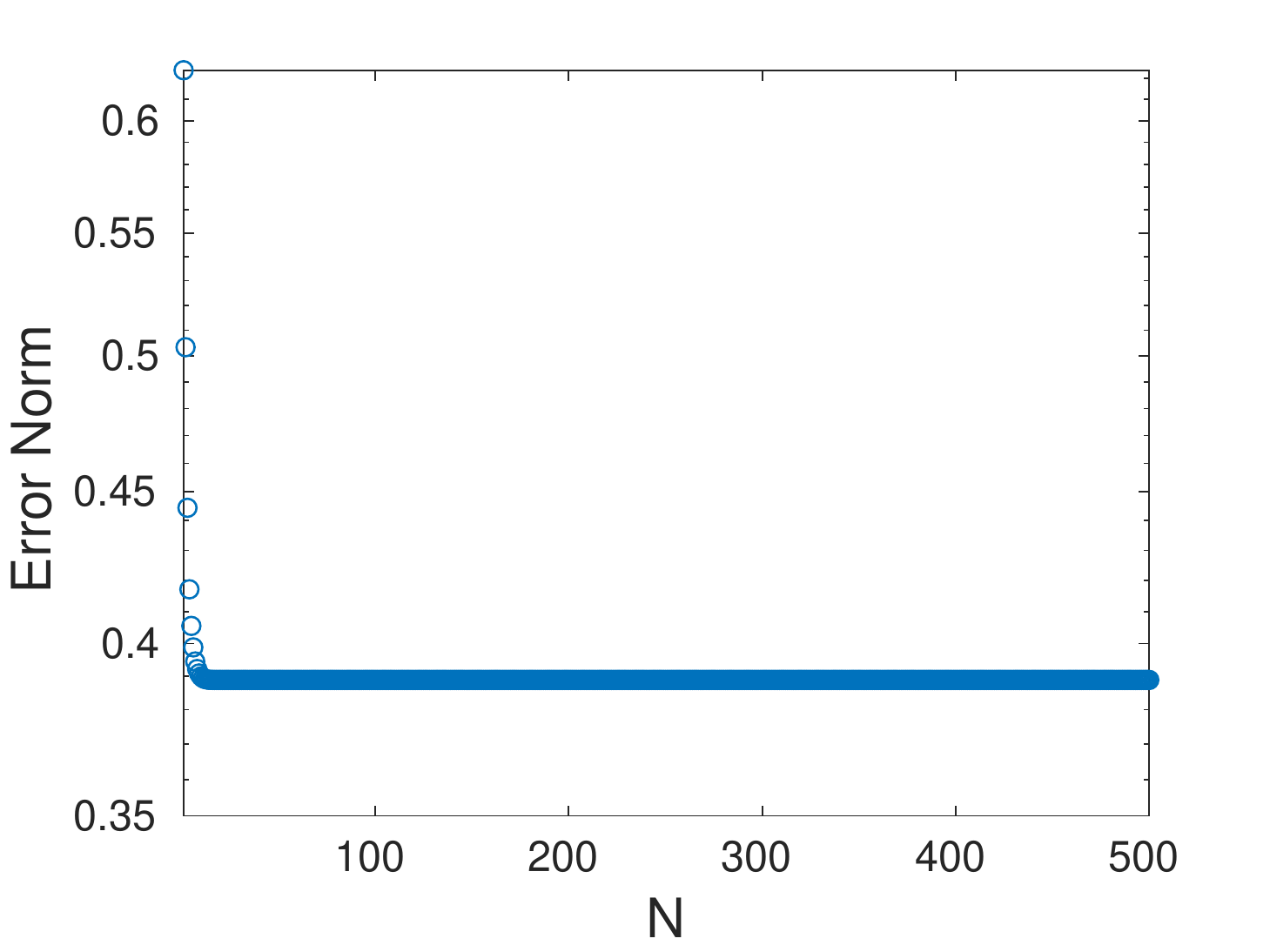}
    \includegraphics[width = 0.32 \textwidth]{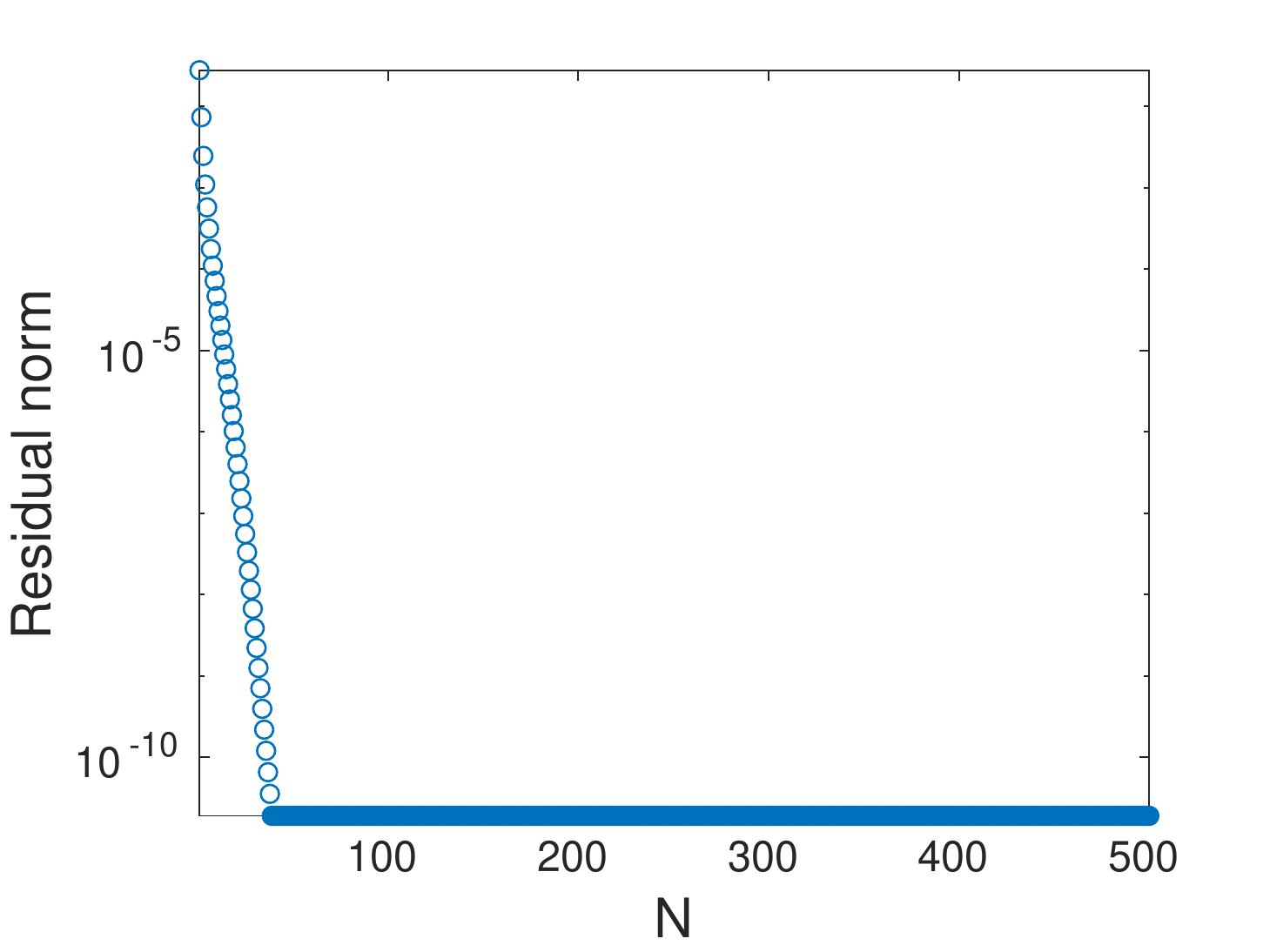}
    \includegraphics[width = 0.32 \textwidth]{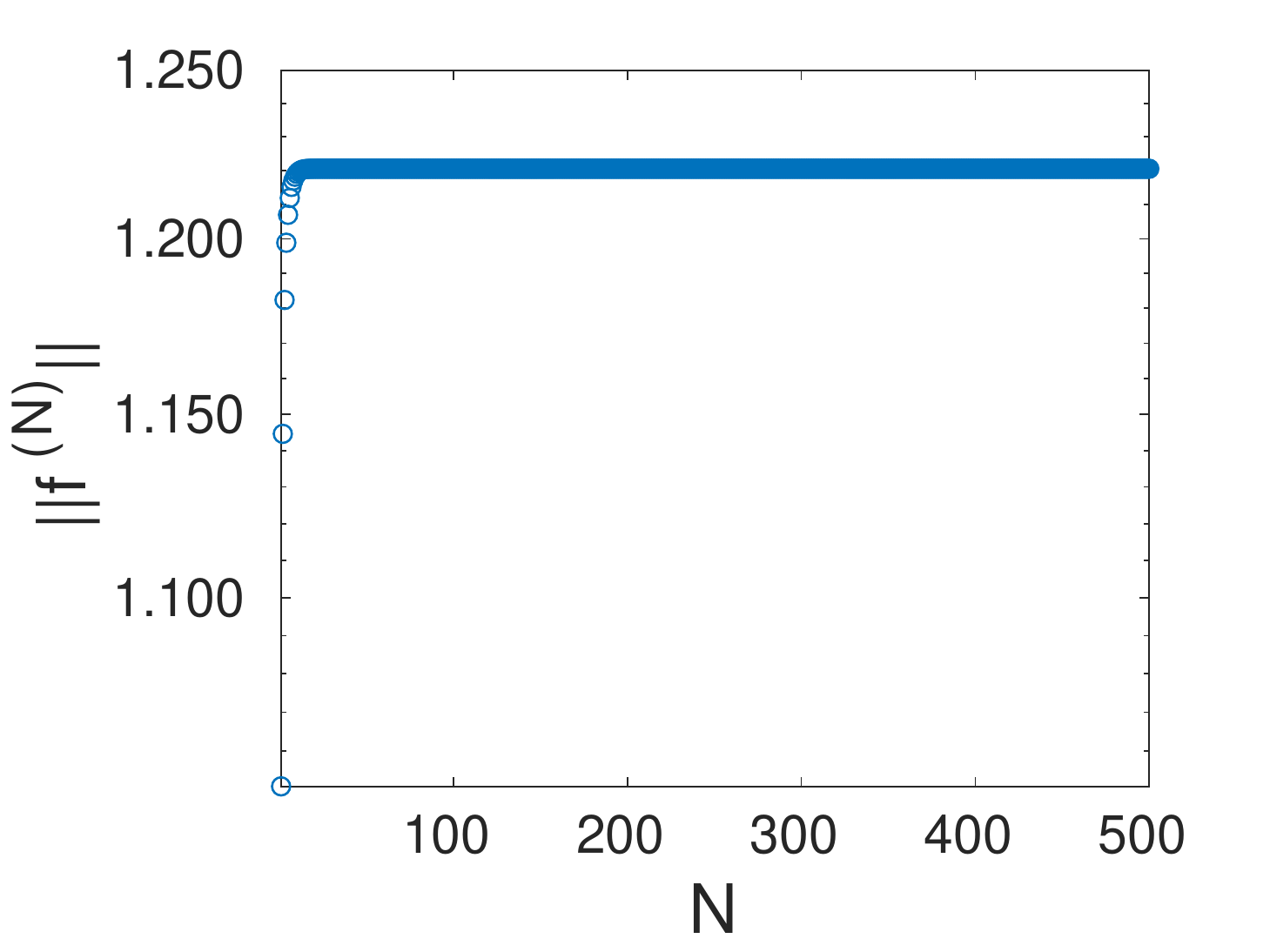}
    \caption{Case $\widetilde{M}$}
  \end{subfigure}
  \begin{subfigure}[b]{\textwidth}
    \includegraphics[width = 0.32 \textwidth]{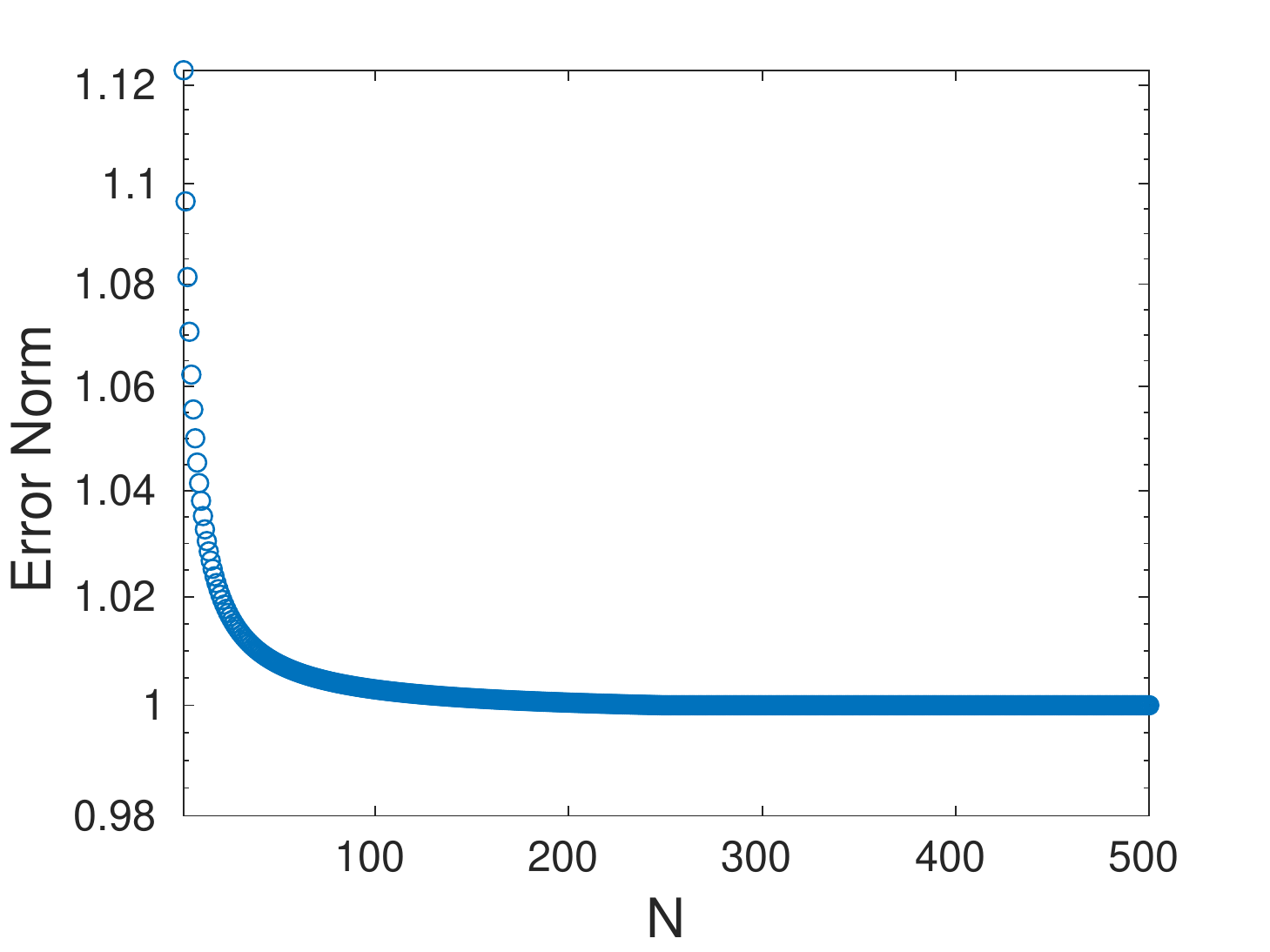}
    \includegraphics[width = 0.32 \textwidth]{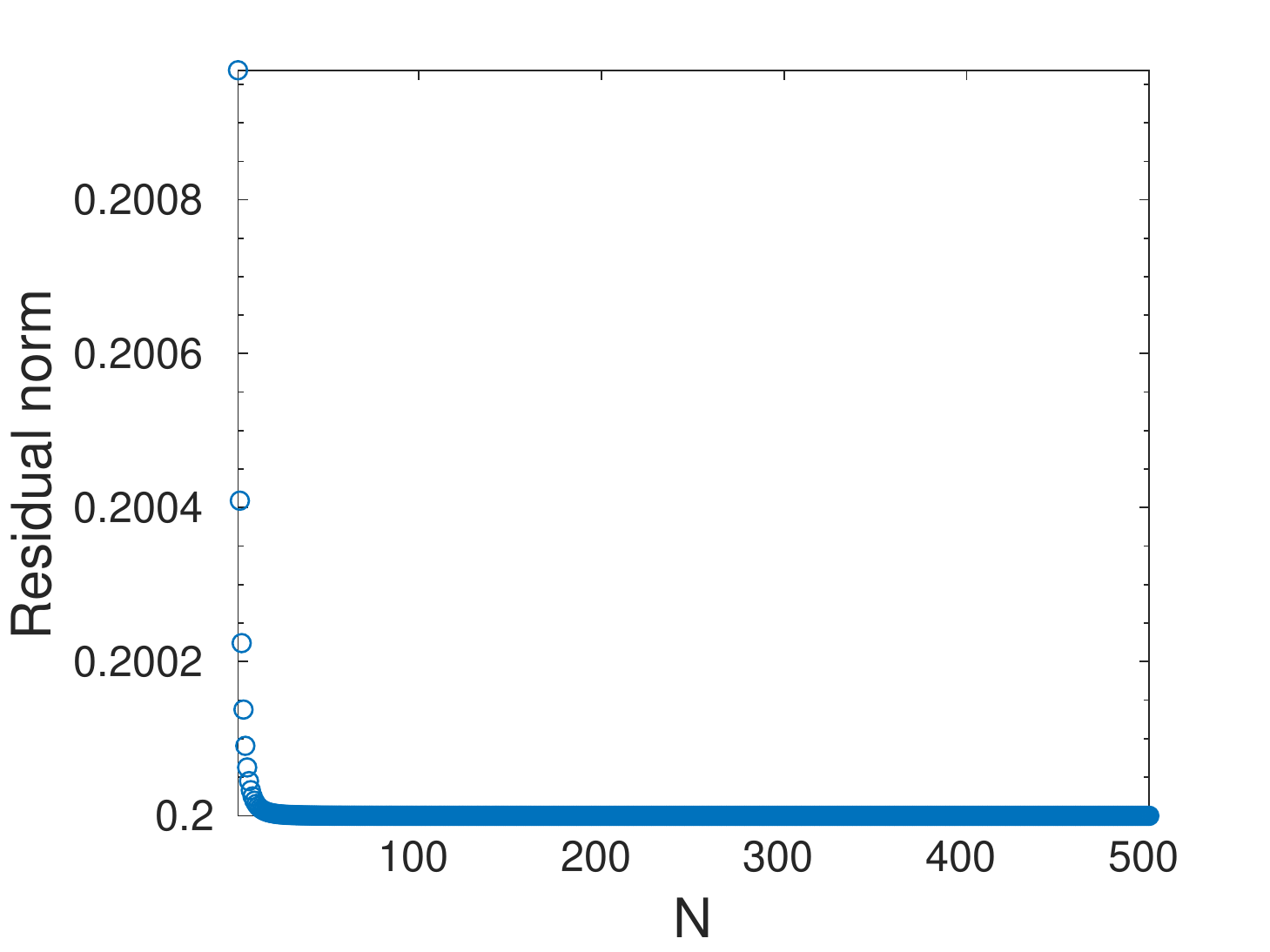}
    \includegraphics[width = 0.32 \textwidth]{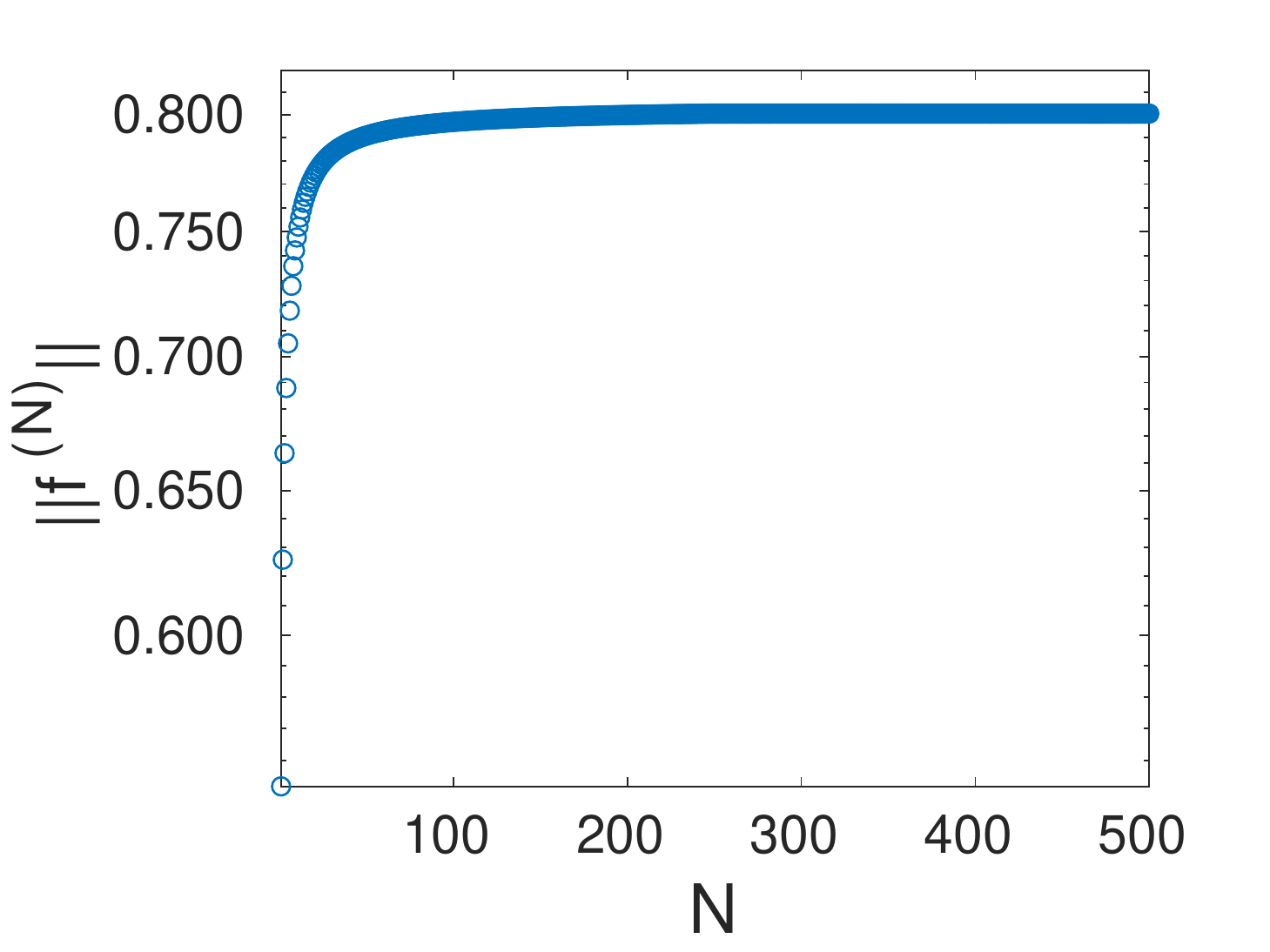}
    \caption{Case $\mathcal{R}$}
  \end{subfigure}
  \begin{subfigure}[b]{\textwidth}
    \includegraphics[width = 0.32 \textwidth]{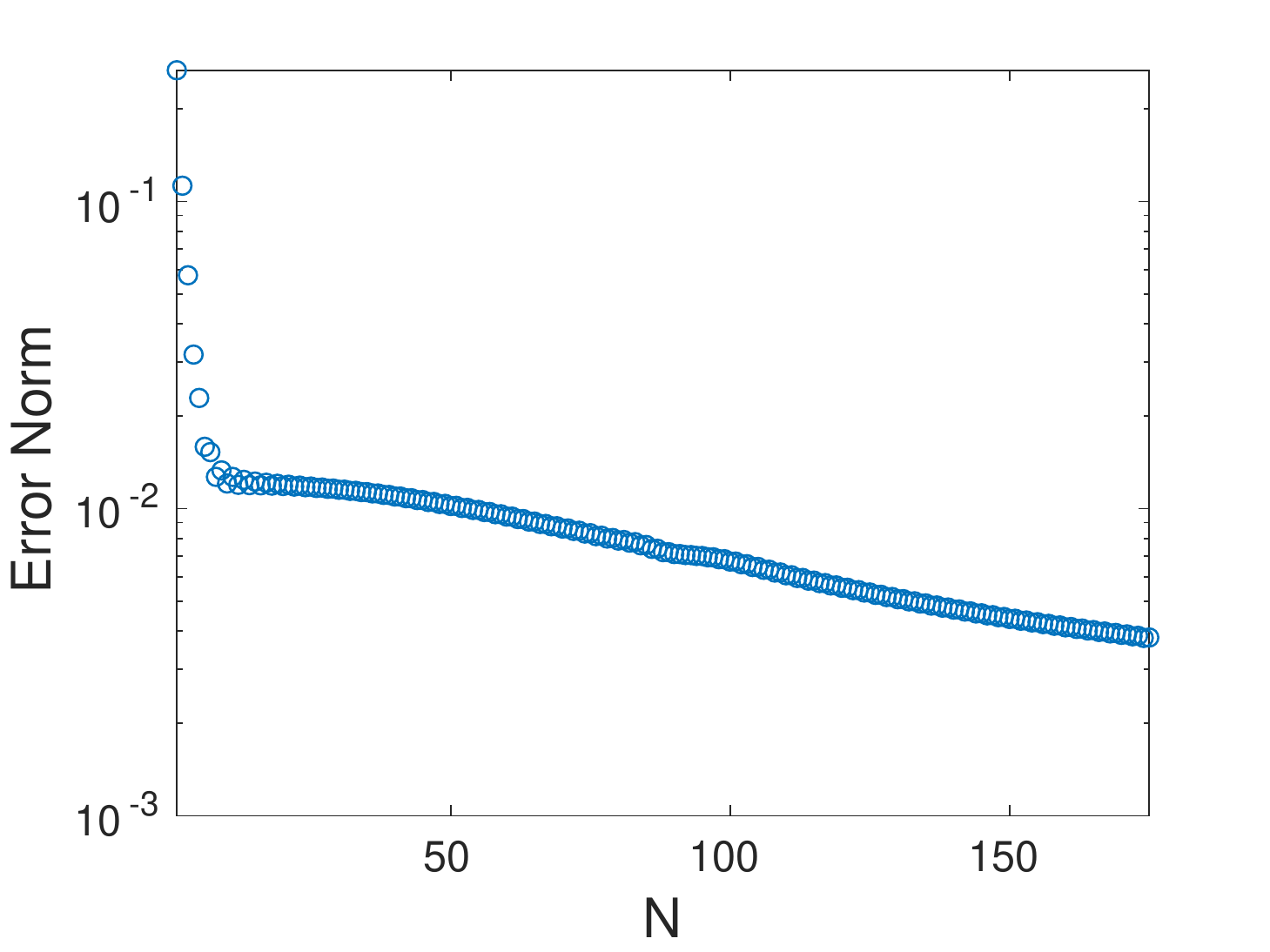}
    \includegraphics[width = 0.32 \textwidth]{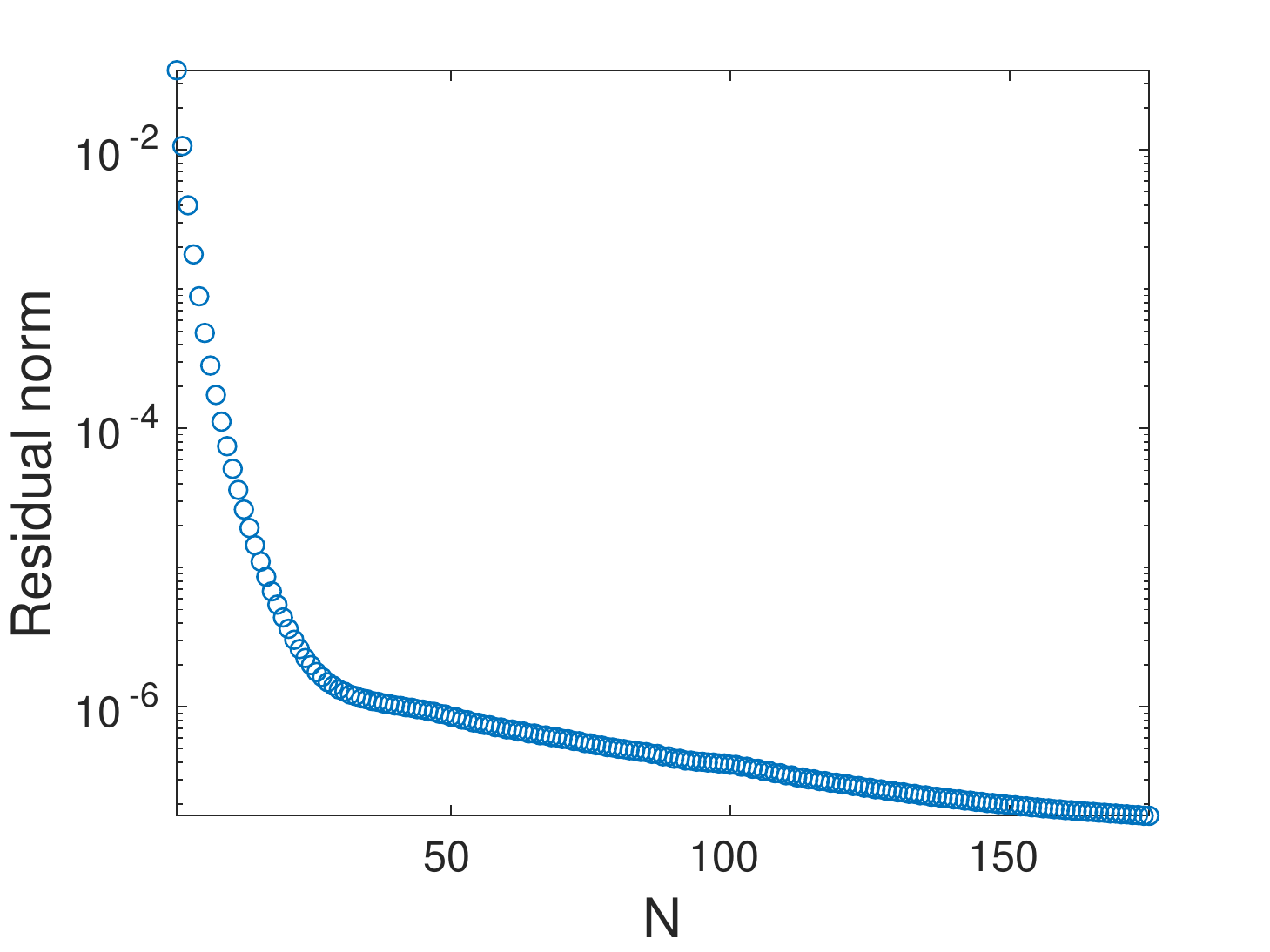}
    \includegraphics[width = 0.32 \textwidth]{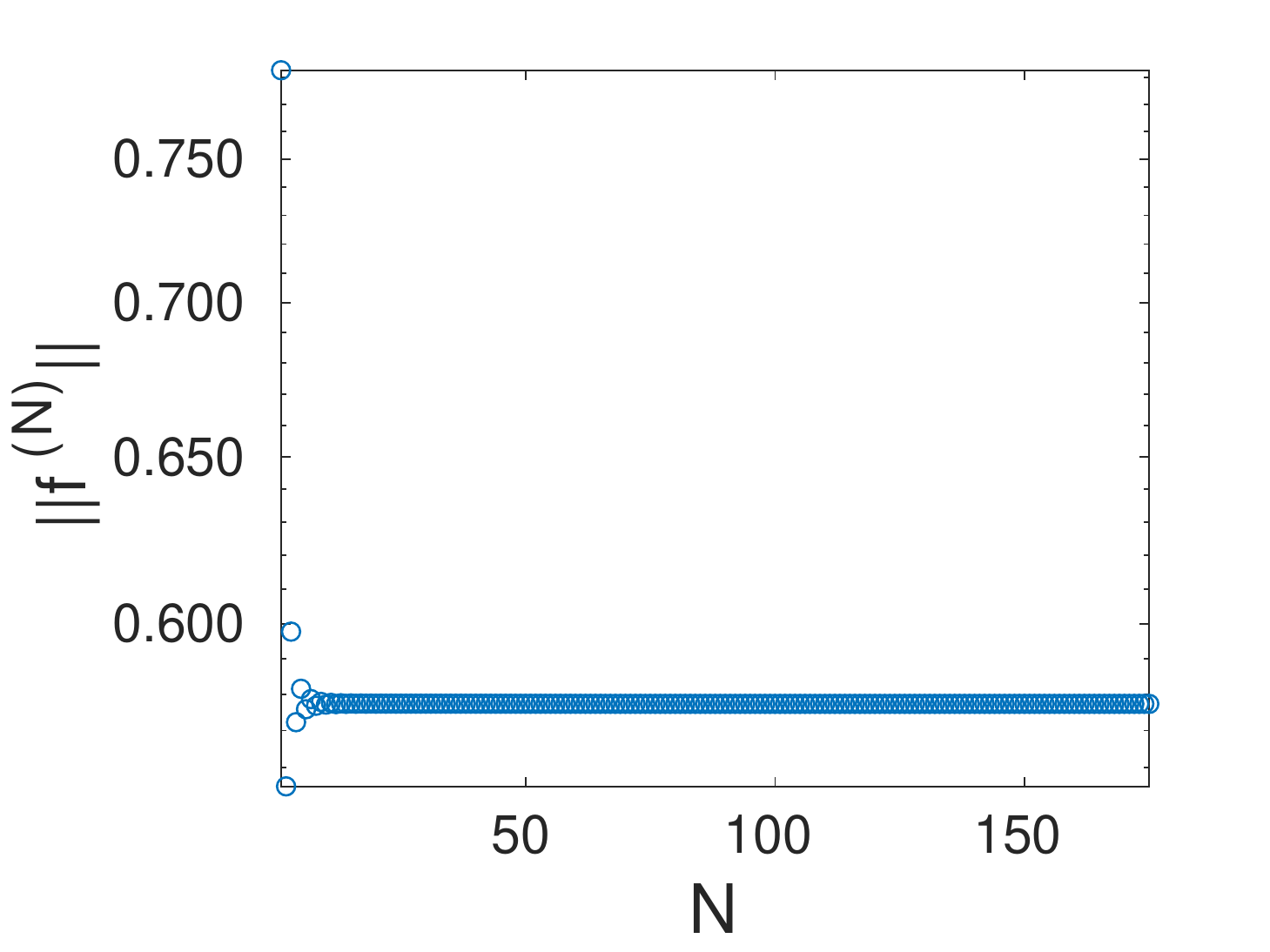}
    \caption{Case $V$}
  \end{subfigure}
  \caption{Error norm and residual norm as a function of iterations for the cases of the injective multiplication operator $M$ (baseline case), the weighted right shift $\mathcal{R}$, the non-injective multiplication operator $\widetilde{M}$, and the Volterra operator $V$.} \label{fig:error4cases}
\end{figure}

\begin{figure}[!t]
  \includegraphics[width = 5.6cm]{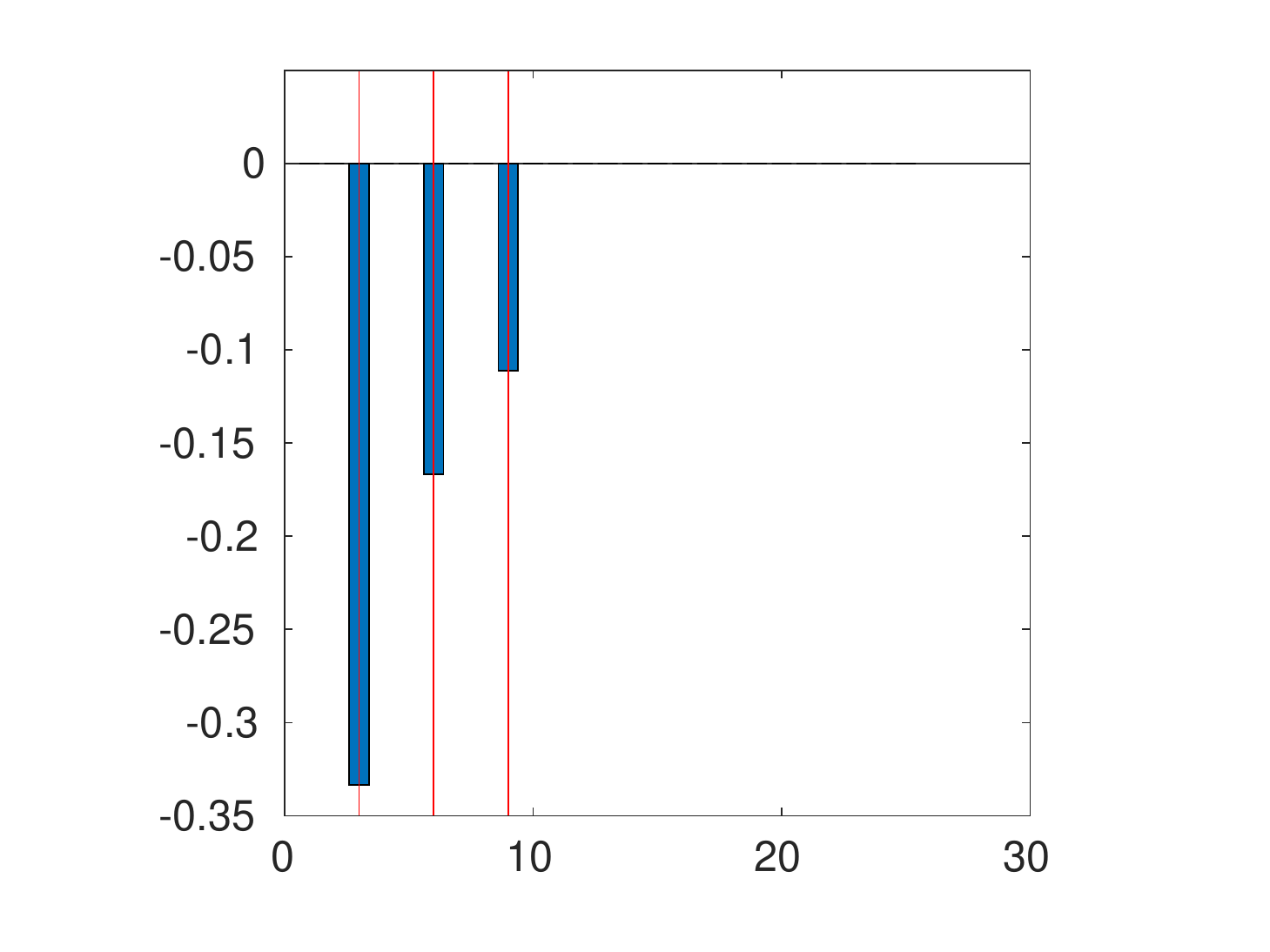}
  \caption{Support of the error vector (blue bars) for the non-injective problem $\widetilde{M}f=g$ at final iteration $N = 500$. The red lines mark the entry positions of the components of the kernel space of $\widetilde{M}$.} \label{fig:noninjkernel}
\end{figure}

The asymptotics $\|f-\widehat{f^{(N)}}\|_{\ell^2}\to 1.0$ and $\|g-\mathcal{R}\widehat{f^{(N)}}\|_{\ell^2}\to 0.2$ found numerically for the problem $\mathcal{R}f=g$ can be understood as follows. Since $\widehat{f^{(N)}}\in\mathcal{K}(\mathcal{R},g)$ and since the latter subspace only contains vectors with zero component along $e_1$, the error vector $\mathscr{E}_N=f-\widehat{f^{(N)}}$ tends to approach asymptotically the vector $e_1$ that gives the first component of $f=(1,\frac{1}{2},\frac{1}{3},\dots)$, and this explains $\|\mathscr{E}_N\|_{\ell^2}\to 1$.

Analogously, since by construction $g=(0,\frac{1}{5},\frac{1}{20},\frac{1}{45},\dots)$, and since the asymptotics on $\mathscr{E}_N$ implies that each component of $\widehat{f^{(N)}}$ \emph{but the first one} converges to the corresponding component of $f$, then $\widehat{f^{(N)}}\approx(0,\frac{1}{2},\frac{1}{3},\dots)$ for large $N$, whence also  $\mathcal{R}\widehat{f^{(N)}}\approx(0,0,\frac{1}{20},\frac{1}{45},\dots)$. Thus $g$ and $\mathcal{R}\widehat{f^{(N)}}$ tend to differ by only the vector $\frac{1}{5}e_2$, which explains $\| \mathfrak{R}_N \|_{\ell^2}\to\frac{1}{5}$.

In fact, the lack of norm vanishing of error and residual for the problem $\mathcal{R}f=g$ is far from meaning that the approximants $\widehat{f^{(N)}}$ carry no information about the exact solution $f$: in complete analogy to what we discussed in a more general context in \cite[Sect.~3 and Sect.~4]{CMN-truncation-2018} -- in particular in \cite[Theorems 3.2 and 4.1]{CMN-truncation-2018} -- $\widehat{f^{(N)}}$ reproduces $f$ \emph{component-wise} for all components but the first.

To summarise the above findings, the Krylov-solvable infinite-dimensional problems ($Mf=g$, $Vf=g$) display good (i.e., norm-) convergence of error and residual, which is sharper for the multiplication operator $M$ and quite slower for the Volterra operator $V$, indicating that the choice of the Krylov bases is not equally effective for the two problems. This is in contrast with the non-Krylov-solvable problem ($\mathcal{R}f=g$), which does not converge in norm at all. The uniformity in the size of the solutions produced by the GMRES algorithm appears not to be affected by the presence or the lack of Krylov-solvability.
%
%

\subsection{Lack of injectivity}~

We then focussed on the behaviour of the truncated problems in the absence of injectivity, by means of the case study operator $\widetilde{M}$ defined by \eqref{eq:diagonal-noninjective}.

Let us observe that the inverse problem $\widetilde{M} f=g$, with $g\in\mathrm{ran}\widetilde{M}$, admits an \emph{infinity} of solutions, yet even in the lack of injectivity Corollary \ref{cor:self-adj_Kry} guarantees that such a problem admits a \emph{unique} Krylov solution.

Numerically we found the following.
\begin{itemize}
 \item As opposite to the baseline case $M$, the infinite-dimensional error norm  $\|\mathscr{E}_N\|_{\ell^2}=\|f-\widehat{f^{(N)}}\|_{\ell^2}$ does not vanish with the truncation size and remains instead uniformly bounded. The infinite-di\-men\-sional residual norm $\Vert \widetilde{M} \widehat{f^{(N)}} - g \Vert_{\ell^2(\mathbb{N})}$, instead, displays the same vanishing behaviour as for $M$
 (Fig.~\ref{fig:error4cases}).
 \item The norm of the approximated solution $\Vert \widehat{f^{(N)}} \Vert_{\ell^2(\mathbb{N}}$ remains uniformly bound\-ed (Fig.~\ref{fig:error4cases}).
\end{itemize}

The reason as to the observed lack of convergence of the error is unmasked in Figure~\ref{fig:noninjkernel}. There one can see that the only components in the error vector that are non-zero are the components corresponding to kernel vector entries.

This shows that the Krylov algorithm has indeed found a solution to the problem, modulo the kernel components in $f$.

\appendix

\section{Some prototypical example operators}

Let us review in this Appendix certain operators in Hilbert space that were useful in the course of our discussion, both as a source of examples or counter-examples, and as a playground to understand certain mechanisms typical of the infinite dimensionality.

\subsection{The multiplication operator on $\ell^2(\mathbb{N})$}\label{sec:multiplication_op}~

Let us denote with $(e_n)_{n\in\mathbbm{N}}$ the canonical orthonormal basis of $\ell^2(\mathbb{N})$. For a given bounded sequence $a\equiv(a_n)_{n\in\mathbbm{N}}$ in $\mathbb{C}$, the multiplication by $a$ is the operator $M^{(a)}:\ell^2(\mathbb{N})\to\ell^2(\mathbb{N})$ defined by $M^{(a)}e_n=a_n e_n$ $\forall n\in\mathbb{N}$ and then extended by linearity and density, in other words the operator given by the series
\begin{equation}
 M^{(a)}\;=\;\sum_{n=1}^\infty a_n|e_{n}\rangle\langle e_n|
\end{equation}
(that converges strongly in the operator sense).

$M^{(a)}$ is bounded with norm $\|M^{(a)}\|_{\mathrm{op}}=\sup_n|a_n|$ and spectrum $\sigma(M^{(a)})$ given by the closure in $\mathbb{C}$ of the set $\{a_1,a_2,a_3\dots\}$. Its adjoint is the multiplication by $a^*$. Thus, $M^{(a)}$ is normal. $M^{(a)}$ is self-adjoint whenever $a$ is real and it is compact if $\lim_{n\to\infty}a_n=0$.

\subsection{The right-shift operator on $\ell^2(\mathbb{N})$}\label{sec:Rshift}~

The operator $R:\ell^2(\mathbb{N})\to\ell^2(\mathbb{N})$ defined by $Re_n=e_{n+1}$ $\forall n\in\mathbb{N}$ and then extended by linearity and density, in other words the operator given by the series
\begin{equation}
 R\;=\;\sum_{n=1}^\infty |e_{n+1}\rangle\langle e_n|
\end{equation}
(that converges strongly in the operator sense), is called the right-shift operator.

$R$ is an isometry (i.e., it is norm-preserving) with closed range $\mathrm{ran} R=\{e_1\}^\perp$. In particular, it is bounded with $\|R\|_{\mathrm{op}}=1$, yet not compact, it is injective, and  invertible on its range, with bounded inverse
\begin{equation}
 R^{-1}:\mathrm{ran} R\to\cH\,,\qquad R^{-1}\;=\;\sum_{n=1}^\infty |e_n\rangle\langle e_{n+1}|\,.
\end{equation}

The adjoint of $R$ on $\cH$ is the so-called left-shift operator, namely the everywhere defined and bounded operator $L:\cH\to\cH$ defined by the (strongly convergent, in the operator sense) series
\begin{equation}
 L\;=\;\sum_{n=1}^\infty |e_n\rangle\langle e_{n+1}|\,,\qquad L=R^*\,.
\end{equation}
Thus, $L$ inverts $R$ on $\mathrm{ran} R$, i.e., $LR=\mathbbm{1}$, yet $RL=\mathbbm{1}-|e_1\rangle\langle e_1|$. One has $\ker R^*=\mathrm{span}\{e_1\}$.

$R$ and $L$ have the same spectrum $\sigma(R)=\sigma(L)=\{z\in\mathbb{C}\,|\,|z|\leqslant 1\}$, but $R$ has no eigenvalue, whereas the eigenvalue of $L$ form the open unit ball $\{z\in\mathbb{C}\,|\,|z|< 1\}$.

\subsection{The compact (weighted) right-shift operator on $\ell^2(\mathbb{N})$}\label{sec:compactRshift}~

This is the operator $\mathcal{R}:\ell^2(\mathbb{N})\to \ell^2(\mathbb{N})$ defined by the operator-norm convergent series
\begin{equation}\label{eq:compactRshift}
 \mathcal{R}\;=\;\sum_{n=1}^\infty\sigma_n|e_{n+1}\rangle\langle e_n|\,,
\end{equation}
where  $\sigma\equiv(\sigma_n)_{n\in\mathbbm{N}}$ is a given bounded sequence with $0<\sigma_{n+1}<\sigma_n$ $\forall n\in\mathbb{N}$ and $\lim_{n\to\infty}\sigma_n=0$. Thus, $\mathcal{R}e_n=\sigma_n e_{n+1}$.

$\mathcal{R}$ is injective and compact, and \eqref{eq:compactRshift} is its singular value decomposition, with norm $\|\mathcal{R}\|_{\mathrm{op}}=\sigma_1$, $\overline{\mathrm{ran}\,\mathcal{R}}=\{e_1\}^\perp$, and adjoint
\begin{equation}
 \mathcal{R}^*\;=\;\mathcal{L}\;=\;\sum_{n=1}^\infty\sigma_n|e_n\rangle\langle e_{n+1}|\,.
\end{equation}
Thus, $\mathcal{L}\mathcal{R}=M^{(\sigma^2)}$, the operator of multiplication by $(\sigma_n^2)_{n\in\mathbb{N}}$, whereas $\mathcal{R}\mathcal{L}=M^{(\sigma^2)}-\sigma_1^2|e_1\rangle\langle e_1|$.

\subsection{The compact (weighted) right-shift operator on $\ell^2(\mathbb{Z})$}\label{sec:compactRshift-Z}~

This is the operator $\mathcal{R}:\ell^2(\mathbb{Z})\to \ell^2(\mathbb{Z})$ defined by the operator-norm convergent series
\begin{equation}\label{eq:compactRshift-Z}
 \mathcal{R}\;=\;\sum_{n\in\mathbb{Z}}\sigma_{|n|}\,|e_{n+1}\rangle\langle e_n|\,,
\end{equation}
where  $\sigma\equiv(\sigma_n)_{n\in\mathbbm{N}_0}$ is a given bounded sequence with $0<\sigma_{n+1}<\sigma_n$ $\forall n\in\mathbb{N}_0$ and $\lim_{n\to\infty}\sigma_n=0$. Thus, $\mathcal{R}e_n=\sigma_{|n|} e_{n+1}$.

$\mathcal{R}$ is injective and compact, with $\mathrm{ran}\,\mathcal{R}$ dense in $\cH$ and norm $\|\mathcal{R}\|_{\mathrm{op}}=\sigma_0$.  \eqref{eq:compactRshift-Z} gives the singular value decomposition. The adjoint of $\mathcal{R}$ is
\begin{equation}
 \mathcal{R}^*\;=\;\mathcal{L}\;=\;\sum_{n\in\mathbb{Z}}\sigma_{|n|}\,|e_n\rangle\langle e_{n+1}|\,.
\end{equation}
Thus, $\mathcal{L}\mathcal{R}=M^{(\sigma^2)}=\mathcal{R}\mathcal{L}$.

The `inverse of $\mathcal{R}$ on its range' is the densely defined, surjective, unbounded operator $\mathcal{R}^{-1}:\mathrm{ran}\,\mathcal{R}\to \cH$ acting as
\begin{equation}
 \mathcal{R}^{-1}\;=\;\sum_{n\in\mathbb{Z}}\,\frac{1}{\sigma_{|n|}}\,|e_n\rangle\langle e_{n+1}|
\end{equation}
as a series that converges on $\mathrm{ran}\,\mathcal{R}$ in the strong operator sense.

\subsection{The Volterra operator on $L^2[0,1]$}\label{sec:Volterra}~

This is the operator $V:L^2[0,1]\to L^2[0,1]$ defined by
\begin{equation}\label{eq:defVolterra}
 (Vf)(x)\;=\;\int_0^x \!f(y)\,\ud y\,,\qquad x\in[0,1]\,.
\end{equation}

$V$ is compact and injective with spectrum $\sigma(V)=\{0\}$ (thus, the spectral point $0$ is not an eigenvalue) and norm $\|V\|_{\mathrm{op}}=\frac{2}{\pi}$. It's adjoint $V^*$ acts as
\begin{equation}
 (V^*f)(x)\;=\;\int_x^1 \!f(y)\,\ud y\,,\qquad x\in[0,1]\,,
\end{equation}
therefore $V+V^*$ is the rank-one orthogonal projection
\begin{equation}
 V+V^*\;=\;|\mathbf{1}\rangle\langle\mathbf{1}|
\end{equation}
onto the function $\mathbf{1}(x)=1$.

The singular value decomposition of $V$ is
\begin{equation}
 V\;=\;\sum_{n=0}^\infty\sigma_n|\psi_n\rangle\langle\varphi_n|\,,\qquad\quad 
 \begin{array}{rl}
  \sigma_n\;=&\!\frac{2}{(2n+1)\pi} \\
  \varphi_n(x)\;=&\!\sqrt{2}\,\cos\frac{(2n+1)\pi}{2}x \\
  \psi_n(x)\;=&\!\sqrt{2}\,\sin\frac{(2n+1)\pi}{2}x\,,
 \end{array}
\end{equation}
where both $(\varphi_n)_{n\in\mathbb{N}_0}$ and $(\psi_n)_{n\in\mathbb{N}_0}$ are orthonormal bases of $L^2[0,1]$.

Thus, $\mathrm{ran} V$ is dense, but strictly contained in $\cH$: for example, $\mathbf{1}\notin\mathrm{ran}V$. (Observe, though, that the dense subspace of the polynomials on $[0,1]$ is mapped by $V$ onto the non-dense $\mathrm{span}\{x,x^2,x^3,\dots\}$.) 

In fact, $V$ is invertible on its range, but does not have (everywhere defined) bounded inverse; yet $V-z\mathbbm{1}$ does, for any $z\in\mathbb{C}\setminus\{0\}$ (recall that $\sigma(V)=\{0\}$), and
\begin{equation}
 (z\mathbbm{1}-V)^{-1}\psi\;=\;z^{-1}\psi+z^{-2}\!\int_0^x e^{\frac{x-y}{z}}\,\psi(y)\,\ud y\qquad\forall\psi\in\cH\,,\; z\in\mathbb{C}\setminus\{0\}\,.
\end{equation}

The explicit action of the powers of $V$ is
\begin{equation}\label{eq:powersVolterra}
 (V^n f)(x)\;=\;\frac{1}{\,(n-1)!}\int_0^x(x-y)^{n-1}f(y)\,\ud y\,,\qquad n\in\mathbb{N}\,.
\end{equation}

\subsection{The multiplication operator over $\Omega\subset\mathbb{C}$ in $L^2(\Omega)$}\label{sec:mult_annulus}~

This is the operator $M:L^2(\Omega)\to L^2(\Omega)$, $f\mapsto zf$, where $\Omega$ is a bounded open region in $\mathbb{C}$. 
$M_z$ is a normal bounded bijection with norm $\|M_z\|_{\mathrm{op}}=\sup_{z\in\Omega}|z|$, spectrum $\sigma(M_z)=\overline{\Omega}$, and adjoint given by $M_z^*f=\overline{z}f$.

\end{document}